\documentclass[a4paper]{amsart}
\usepackage[textwidth=14cm,top=3cm, bottom=3cm, hcentering]{geometry}
\usepackage[colorlinks=true,linkcolor=black,citecolor=blue]{hyperref}

\usepackage[english]{babel}
\usepackage{amsmath,amscd,amsthm,amssymb,xypic}
\usepackage{color}
\usepackage[utf8]{inputenc}
\normalfont
\usepackage[T1]{fontenc}
\usepackage{tikz-cd}
\usepackage{tikz}
\usepackage{enumerate}
\usepackage{relsize}
\usepackage{comment}
\usepackage{mathtools}
\usepackage{leftindex}
\usepackage{hyperref}
\usepackage{quiver}
\usetikzlibrary{babel}
\usepackage{enumitem}
\usepackage{mathdots}

\usepackage[export]{adjustbox}
\newcommand{\fitText}[2][1]{\resizebox{#1\textwidth}{!}{#2}}

\newcommand{\fitMath}[2][1]{\resizebox{#1\textwidth}{!}{\(\displaystyle #2\)}}

\newcommand\doq{%
  \mathord{%
    \mspace{1mu}%
    \text{\hspace{0.5ex}\Doq}%
    \mspace{1mu}%
  }%
}

\newcommand\Doq{%
    \tikz[line cap=round,x=1ex,y=1ex,line width=0.3pt]
    {\draw (1,0) |- (0,1) (0.55,0);}%
}

\newtheorem{theo}{Theorem}[section]
\newtheorem{lemm}[theo]{Lemma}
\newtheorem{prop}[theo]{Proposition}

\newtheorem*{theorem*}{Theorem}

\theoremstyle{definition}
\newtheorem{defi}[theo]{Definition}
\newtheorem{rema}[theo]{Remark}

\newtheorem{exas}[theo]{Examples}

\newcommand{\CC}{\mathbb{C}}

\newcommand{\RR}{\mathbb{R}}

\newcommand{\ZZ}{\mathbb{Z}}
\newcommand{\Aa}{\mathcal{A}}
\newcommand{\Bb}{\mathcal{B}}
\newcommand{\Cc}{\mathcal{C}}
\newcommand{\Dd}{\mathcal{D}}
\newcommand{\Ee}{\mathcal{E}}

\newcommand{\p}{\mathfrak{p}}
\newcommand{\q}{\mathfrak{q}}

\newcommand{\kk}{\mathbf{k}}

\newcommand{\del}{\partial}
\newcommand{\delb}{{\bar \partial}}

\newcommand{\A}{H_{\Aa}}
\newcommand{\BC}{H_{\Bb\Cc}}

\newcommand{\xra}[1]{\xrightarrow{#1}}
\newcommand{\lra}{\longrightarrow}

\newcommand{\ov}{\overline}

\newcommand{\Hom}{\mathrm{Hom}}

\newcommand{\Ch}{\mathrm{Ch}_{\kk}}
\newcommand{\Bico}{\mathrm{BiCo}_{\kk}}
\newcommand{\Bicoc}{\mathrm{BiCo}_{\CC}}

\newcommand{\Sch}{\mathcal{B}}
\newcommand{\Inf}{\mathrm{Inf}}
\newcommand{\Le}{\mathrm{L}}

\newcommand{\ele}{\tikz[baseline=0mm]{ 
\draw[thick] (0,0.015) -- (0.2,0.015); 
\draw[thick]  (0,0) -- (0,0.215);
}}

\newcommand{\pr}{\operatorname{pr}}

\newcommand{\mep}{\text{-}}

\thanks{
 P.M. was supported by FCT - Fundação para a Ciência e Tecnologia, I.P. by project 2021.06151.BD with DOI identifier https://doi.org/10.54499/2021.06151.BD.
 A.S-G. was supported by AGAUR, Govern de Catalunya, through the FI Program.
 Both authors acknowledge financial support from the Spanish State Research Agency through projects PID2020-117971GB-C22, PID2024-155646NB-I00 and EUR2023-143450.
}

\author{Pedro Magalhães}
\author{Anna Sopena-Gilboy}
\title[The inflation functor]{The inflation functor\\ in pluripotential homological algebra
}

\begin{document}

\maketitle

\begin{abstract}
 We introduce a functor from cochain complexes to bicomplexes, called \textit{inflation functor}, which sends quasi-isomorphisms to the class of pluripotential weak equivalences. We show this functor is part of a Quillen adjunction. Its right adjoint is a well-known construction in complex geometry, which gives a sheaf-theoretic presentation of Bott-Chern and Aeppli cohomologies. The inflation functor plays a key role in pluripotential Koszul duality theory for operads and allows us to construct the $\infty$-category of bicomplexes in the pluripotential sense.
    
\end{abstract}

\section{Introduction}
Associated with a bicomplex $(A,\del,\delb)$, also called \textit{double complex} in the literature,  one can consider several cohomology theories, each capturing different aspects of the interaction between the two differentials. First, the total differential
$d=\del+\delb$ gives rise to the total cohomology $H^*(A)$.
In addition, taking cohomology with respect to one differential at a time yields the bigraded Dolbeault-type cohomologies
$H^{*,*}_{\partial}(A)$ and $H^{*,*}_{\delb}(A)$.
The row and column filtrations of the total complex induce two spectral sequences whose first pages are isomorphic to \(H^{*,*}_{\delb}(A)\) and \(H^{*,*}_{\partial}(A)\), respectively, and which converge, under boundedness assumptions, to the total cohomology $H^*(A)$.
Beyond these spectral sequences, the bicomplex structure allows for two further natural cohomology theories: \textit{Bott-Chern} \cite{BC} and \textit{Aeppli} \cite{Aep} cohomologies, defined by
\[
H^{*,*}_{\mathcal{BC}}(A)
:= \frac{\ker\partial \cap \ker\delb}
{\operatorname{im}(\partial\delb)}\quad\text{ and }\quad H^{*,*}_{\mathcal{A}}(A)
:= \frac{\ker(\partial\delb)}
{\operatorname{im}\partial + \operatorname{im}\delb}
.\]
These bigraded cohomologies fit naturally between the Dolbeault-type and total cohomologies and encode finer information about the bicomplex.
Specifically, the identity induces the following commutative diagram
\[
\begin{tikzcd}[ampersand replacement=\&]
	\& {H_{\mathcal{BC}}^{**}(A)} \\
	{H_{\del}^{**}(A)} \& {H^*_{dR}(A)} \& {H_{\delb}^{**}(A)} \\
	\& {H^{**}_{\mathcal{A}}(A)}
	\arrow[from=1-2, to=2-1]
	\arrow[from=1-2, to=2-2]
	\arrow[from=1-2, to=2-3]
	\arrow[Rightarrow, from=2-1, to=2-2]
	\arrow[from=2-1, to=3-2]
	\arrow[from=2-2, to=3-2]
	\arrow[Rightarrow, from=2-3, to=2-2]
	\arrow[from=2-3, to=3-2]
\end{tikzcd}
\]
where the horizontal double arrows denote the row and column spectral sequences mentioned above, also known as the \textit{Frölicher spectral sequences}.

This diagram has been extensively studied in the context of non-K\"ahler geometry. Indeed, a primary source of bicomplexes arises from the complexified de Rham algebra of a complex manifold, which naturally decomposes into forms of type \((p,q)\). It is well known that compact K\"ahler manifolds satisfy the \(\del\delb\)-lemma, which is equivalent to the statement that any (and hence all) of the maps in the above diagram are isomorphisms (see for instance \cite{Angella14}). 
In particular, Bott--Chern and Aeppli cohomologies are especially interesting biholomorphic invariants for non-K\"ahler manifolds. They play important roles in the study of special metrics, such as balanced, Gauduchon, and strong Kähler with torsion (SKT) metrics, and have direct applications in deformation theory and in mathematical physics, for instance in the Hull-Strominger system arising in string theory \cite{GFGS26}. 

Various recent developments in complex geometry are related to the consideration of a new class of weak equivalences on the category of bicomplexes, called \textit{pluripotential weak equivalences} \cite{MiSte24}, \cite{pluri}, \cite{CiLivWhi25}, \cite{CiGaSo25}, \cite{SGtesi}. These are defined by those maps of bicomplexes inducing isomorphisms in both Bott--Chern and Aeppli cohomologies.
Associated with these weak equivalences, there is a finer notion of formality (in the sense of rational homotopy theory), with interesting consequences in complex geometry, both in the Kähler and non-Kähler settings.
This is the subject of several recent works
\cite{AnTo15}, \cite{TaTo17}, \cite{SfeTo22}, \cite{SfeTo24},\cite{Ru25}, \cite{Jonas1}, \cite{Jonas2}.

The category $\mathrm{BiCo}_\kk$ of bicomplexes over a field $\kk$ admits a model structure with pluripotential weak equivalences \cite{pluri}.  
In this note, we introduce the \textit{inflation functor} as part of a further structural development of pluripotential homotopy theory. 
This is a functor 
\[ \mathrm{Inf} : \mathrm{Ch}_\kk \to \mathrm{BiCo}_\kk\] 
from cochain complexes to bicomplexes, which sends quasi-isomorphisms to pluripotential weak equivalences and is conservative.
We show it admits a right adjoint, which we identify with a well-known sheaf-theoretic construction in complex geometry. This construction was originally introduced by Bigolin \cite{zbMATH03320556} (see also   \cite{demailly} and \cite{schweizer}) 
as a functor from complex manifolds to complexes of locally free sheaves. It has the property that Bott-Chern and Aeppli cohomologies arise as cohomologies, in certain degrees, of this complex. The same construction gives 
a functor $\Bb$ from bicomplexes to cochain complexes with similar properties.
Hodge-theoretic aspects of this functor have recently been developed in \cite{StelzigSchRemarks} and \cite{Piovani}, where the authors refer to it as the \textit{Schweitzer complex} and the \textit{Bigolin complex}, respectively. Given that the construction was first introduced by Bigolin, the latter designation is arguably the more natural. We will use $\Bb$ to denote this functor.
Over a field of characteristic zero, the inflation functor allows us to place $\Bb$ within a Quillen adjunction, thereby giving its definition a natural homotopical meaning.
We prove:

\begin{theorem*}
    There is a Quillen adjunction 
    \[
    \begin{tikzcd}[ampersand replacement=\&]
	{\Inf:\Ch} \& {\Bico:\Sch}\;.
	\arrow[shift left, from=1-1, to=1-2]
	\arrow[shift left, from=1-2, to=1-1]
    \end{tikzcd}
    \]
\end{theorem*}
Here, $\Ch$ is endowed with the projective model structure and \(\Bico\) with the pluripotential model structure of \cite{pluri}.

Furthermore, we show that, when restricted to non-positively graded cochain complexes, the inflation functor is oplax symmetric monoidal. Moreover, the natural map
\[
\Inf(C\otimes D)\to \Inf(C)\otimes \Inf(D)
\]
is a pluripotential weak equivalence for all cochain complexes \(C\) and \(D\) in \(\Ch^{\leq 0}\). 
This allows us to show that the functor $\Bb$ restricted to third-quadrant bicomplexes is lax symmetric monoidal.
The inflation functor and its monoidal properties are used in \cite{SGtesi} as a key tool for a theory of
pluripotential Koszul duality for quadratic operads. 
In this note, we use the above results to produce a dg-enrichment of the category of bicomplexes, which leads to an explicit description of the $\infty$-category of bicomplexes localizing at pluripotential weak equivalences. For the applications we have in mind, we restrict our study to the case when the field $\kk$ is of characteristic zero. This gives a model for the inflation functor that enjoys full symmetry with respect to the two differentials, in sync with the pluripotential $A_\infty$-models developed in \cite{SGtesi}. Other models for the inflation functor may be more suitable for complexes over a ring.

\subsection*{Acknowledgments}
We thank Muriel Livernet for her insights on first proof of Theorem~\ref{adj-theo}. 
We also thank Geoffroy Horel and David Martínez-Carpena for discussions on Section~\ref{sec:InftyCat}. 
Finally, we thank Joana Cirici and Jonas Stelzig for being our great advisors (and for many fruitful discussions).

\section{Preliminaries}

We review basic concepts and properties concerning the category of bicomplexes and establish conventions and notation to be used in the subsequent sections. 
Throughout the paper, \(\kk\) will denote a field of characteristic zero.

\begin{defi}
    A \emph{bicomplex} \((A,\del,\delb)\) consists of a bigraded \(\kk\)-vector space \(A=\bigoplus A^{*,*}\) together with two differentials $\del$ and $\delb$ of bidegrees $(1,0)$ and $(0,1)$, respectively, that anti-commute:
    \[\del^2 = 0, \;\,\delb^2 = 0,\;\,\text{ and }
        \del \delb = - \delb \del.
    \]
\end{defi}
    A morphism of bicomplexes is a morphism of vector spaces respecting the bigrading and commuting with both differentials. We denote by $\Bico$ the category of bicomplexes.

The usual tensor product of cochain complexes extends naturally to bicomplexes 
\begin{equation*}
    (A \otimes B)^{p,q} = \bigoplus_{\substack{r+s = p \\ t + u = q}} A^{r,t} \otimes B^{s,u} 
\end{equation*}
making $\Bico$ into a symmetric monoidal category. The category $\Bico$ is also enriched over itself
via the internal Hom
defined by
\[\underline{\Hom}_{\Bico}(A,B)^{p,q}:=\prod_{r,s\in \ZZ} \Hom_{\kk}(A^{p,q},B^{p-r,q-s})\]
with differentials
\[\del f:=\del_{B}f-(-1)^{|f|}f\del_{A}\quad\text{  and  }\quad\delb f:=\delb_{B}f-(-1)^{|f|}f\delb_{A}.\]
Internal Hom and tensor product form an adjunction
\[
(\_ \otimes A):\Bico\leftrightarrows \Bico: \underline{\Hom}(A, \_).
\]

The following construction was introduced by Bigolin \cite{zbMATH03320556} (see also
\cite{demailly} and \cite{schweizer}). 
In our convention, we differ by a shift in degree by one.

\begin{defi}Let $(A,\del,\delb)$ be a bicomplex. For each $(p,q)\in\ZZ^2$, define a cochain complex
 $(\Bb_{p,q}(A),d_{\Bb})$ as follows: as a graded vector space, we let
    \begin{align*}
        \Bb_{p,q}^k(A)&:=\bigoplus_{\substack{r+s=k-1\\r<p,s<q}} A^{r,s} &\text{if }k\leq p+q-1\\
        \Bb_{p,q}^k(A)&:=\bigoplus_{\substack{r+s=k\\r\geq p,s\geq q}} A^{r,s} &\text{if }k\geq p+q.
    \end{align*}
The differential $d_{\Bb}$ is given by:
\[
\cdots\overset{\pr\circ d}{\longrightarrow}\Bb_{p,q}^{p+q-2}(A)\overset{\pr\circ d}{\longrightarrow}\Bb_{p,q}^{p+q-1}(A)\overset{\del\delb}{\longrightarrow}\Bb_{p,q}^{p+q}(A)\overset{d}{\longrightarrow}\Bb_{p,q}^{p+q+1}(A)\overset{d}{\longrightarrow}\cdots
,\]
where \(\pr\) denotes projection from the bicomplex \(A\) to the direct summand \(\Bb_{p,q}^k(A)\) and \(d=\del+\delb\).
\end{defi}

By construction, we have isomorphisms
\[
H^{p,q}_{\mathcal{BC}}(A)\cong H^{p+q}(\Bb_{p,q}^*(A))
\quad \text{
and }\quad
H^{p,q}_{\mathcal{A}}(A)\cong H^{p+q-1}(\Bb_{p,q}^*(A)).
\]

Pictorially, the complex $\Bb_{p,q}$ looks as follows:
\[
\fitText{
\begin{tikzcd}[ampersand replacement=\&]
	\&\&\& {A^{p,q+2}} \& {A^{p+1,q+2}} \& \iddots \\
	\&\&\& {A^{p,q+1}} \& {A^{p+1,q+1}} \& {A^{p+2,q+1}} \\
	\&\&\& {A^{p,q}} \& {A^{p+1,q}} \& {A^{p+2,q}} \\
	{A^{p-3,q-1}} \& {A^{p-2,q-1}} \& {A^{p-1,q-1}} \&\& { p+q} \& {p+q+1} \&  \dots \\
	{A^{p-3,q-2}} \& {A^{p-2,q-2}} \& {A^{p-1,q-2}} \& {p+q-1} \\
	\iddots \& {A^{p-2,q-3}} \& {A^{p-1,q-3}} \& {p+q-2} \\
	\&\&\& \vdots
	\arrow[from=1-4, to=1-5]
	\arrow[from=2-4, to=1-4]
	\arrow[from=2-4, to=2-5]
	\arrow[from=2-5, to=1-5]
	\arrow[from=2-5, to=2-6]
	\arrow[from=3-4, to=2-4]
	\arrow[from=3-4, to=3-5]
	\arrow[from=3-5, to=2-5]
	\arrow[from=3-5, to=3-6]
	\arrow[from=3-6, to=2-6]
	\arrow[from=4-1, to=4-2]
	\arrow[from=4-2, to=4-3]
	\arrow["{\del\delb}", from=4-3, to=3-4]
	\arrow[from=5-1, to=4-1]
	\arrow[from=5-1, to=5-2]
	\arrow[from=5-2, to=4-2]
	\arrow[from=5-2, to=5-3]
	\arrow[from=5-3, to=4-3]
	\arrow[from=6-2, to=5-2]
	\arrow[from=6-2, to=6-3]
	\arrow[from=6-3, to=5-3]
\end{tikzcd}}
\]

\begin{rema}
    Let \(A\) be a bicomplex. Denote by \(A[p,q]\) the bicomplex given by
    \[(A[p,q])^{r,s}=A^{r-p,s-q}.\]
    Note that \[\Bb_{p,q}(A)=\Bb_{0,0}(A[-p,-q]).\]
\end{rema}
From now on, \(\Sch\) will denote \(\Bb_{0,0}\).
\begin{defi}
    A morphism of bicomplexes $f : A \to B$ is said to be a \textit{pluripotential weak equivalence} if $H_{\mathcal{BC}}(f)$ and $H_{\mathcal{A}}(f)$ are both isomorphisms.
\end{defi}

There is a pluripotential notion of homotopy between morphisms of bicomplexes, which will be of interest:
\begin{defi}
    \label{hombico-defi}
    Let \(f,g\colon A\to B\) be two morphisms of bicomplexes. A \textit{pluripotential homotopy} from $f$ to \(g\) is a linear map $h\colon A \to B$ of bidegree $(-1,-1)$ such that 
    \begin{equation*}
        [\del, [\delb, h]] = f - g.
    \end{equation*}
    If there exists a homotopy from $f$ to $g$ we write $f \sim g$. A morphism $f$ in $\Hom_{\Bico}(A,B)$ is said to be a \textit{pluripotential homotopy equivalence} if there exists $g$ in $\Hom_{\Bico}(B,A)$ such that $fg \sim id_B$ and $gf \sim id_A$. 
\end{defi}

By Theorem $1.21$ of \cite{pluri}, the class of pluripotential weak equivalences coincides with the class of homotopy equivalences. Moreover, there is a cofibrantly generated model structure on $\Bico$ whose weak equivalences are the pluripotential weak equivalences. The cofibrations are the injective morphisms, and the fibrations are the surjective morphisms. We will refer to this as the \textit{pluripotential model structure}. The tensor product and internal Hom endow \(\Bico\) with the structure of a symmetric monoidal model category, and hence \(\Bico\) is a \(\Bico\)-enriched model category.

\begin{lemm}\label{lemm:SchWeakEq}
Consider on $\Bico$ the pluripotential model structure and on $\Ch$
the projective model structure.
    The functor $\Sch\colon \Bico \to \Ch$ preserves weak equivalences and fibrations.  
\end{lemm}
\begin{proof}
    Corollary 1.25 of \cite{pluri} states that if \(f\colon A \to B\) is a pluripotential weak equivalence, then \(H(f)\) is an isomorphism for any additive functor \(H \colon \Bico \to V,\)
    where \(V\) is an additive category and \(H(P)=0\) for every acyclic bicomplex \(P\); that is,
    \(\A(P)=\BC(P)=0\) (see Definition 1.23 in \cite{pluri}).
    
    By \cite{StelzigSchRemarks}, the functor \(H^* \circ \Sch\), where \(H^*\) denotes taking cohomology of a cochain complex, satisfies this condition. It follows that \(\Sch\) preserves weak equivalences. Moreover, in both model categories \(\Bico\) and \(\Ch\), fibrations are simply (bi)degree-wise surjections. Since \(\Sch\), viewed as a graded vector space, is given by the totalization of the restricted bicomplex, it preserves surjections.

\end{proof}

It thus follows that if $\Sch$ is a right adjoint, then there is a Quillen adjunction between $\Bico$ and $\Ch$. This is the content of the next section.

\section{The inflation functor}

In this section we define the \emph{inflation functor} from cochain complexes to bicomplexes and show its right adjoint is $\Bb$. Moreover, we prove that the inflation functor and $\Bb$ form a Quillen pair.

Let us consider the following family of bicomplexes, $\{\ele_n\}_{n\in \ZZ}$, where
\begin{equation}
    \label{eles-eq}
    \ele_n^{ij}=
\begin{cases}
    \kk \text{ if } i,j,n \geq 0 \text{ and } i + j =  n \\
    \kk \text{ if } i,j > 0, n\geq0\text{ and } i+j =  n+1,\\
    \kk\text{ if } i,j \leq 0, n < 0 \text{ and } i+j = n\\
    \kk\text{ if } i,j,n < 0 \text{ and } i+j = n-1\\
    0\text{ otherwise}.
\end{cases}
\end{equation}
Denoting the unit of \(\kk\subseteq \ele_n\) 
sitting in bidegree $(i,j)$ by $(i,j)_n$, define
\begin{equation}
    \label{eledif-eq}
    \del(i,j)_n=j(i+1,j)_n\;\;\text{   and  }\;\;\delb(i,j)_n=-i(i,j+1)_n.
\end{equation}

\begin{exas}
We show some examples for different values of \(n\).

\begin{itemize}[leftmargin=*]
    \item For \(|n|=1\) we have:
    \[
    \ele_1=
    \begin{tikzcd}
        \kk \arrow[r, "\cdot1"] & \kk \\
          & \kk \arrow[u,"\cdot(-1)"]
        \end{tikzcd}\qquad\qquad\text{and}\qquad\qquad
    \ele_{-1}=
    \begin{tikzcd}
        \kk  &   \\
         \kk \arrow[u, "\cdot1"] \arrow[r,"\cdot(-1)"]& \kk\;.
    \end{tikzcd}
    \]
    \item For \(|n|=2\) we have:
    \[     
    \ele_{2}=
     \begin{tikzcd}[ampersand replacement=\&]
    	\kk \& \kk \\
    	\& \kk \& \kk \\
    	\&\& \kk
    	\arrow["{\cdot 2}", from=1-1, to=1-2]
    	\arrow["{\cdot(-1)}", from=2-2, to=1-2]
    	\arrow["{\cdot 1}", from=2-2, to=2-3]
    	\arrow["{\cdot(-2)}", from=3-3, to=2-3]
    \end{tikzcd}\qquad\text{and}\qquad
    \ele_{-2}=
    \begin{tikzcd}
        \kk  &  &  \\
         \kk \arrow[u, "\cdot2"] \arrow[r,"\cdot(-1)"]& \kk & \\
          & \kk \arrow[r,"\cdot(-2)"] \arrow[u,"\cdot1"] & \kk\;.
    \end{tikzcd}
    \]
    \item For \(n=-3\) we have:
    \[
    \ele_{-3}=
    \begin{tikzcd}[ampersand replacement=\&]
	\kk \\
	\kk \& \kk \\
	\& \kk \& \kk \\
	\&\& \kk \& \kk\;.
	\arrow["\cdot3", from=2-1, to=1-1]
	\arrow["{\cdot(-1)}"', from=2-1, to=2-2]
	\arrow["\cdot2", from=3-2, to=2-2]
	\arrow["{\cdot(-2)}"', from=3-2, to=3-3]
	\arrow["\cdot1", from=4-3, to=3-3]
	\arrow["{\cdot(-3)}"', from=4-3, to=4-4]
    \end{tikzcd}
    \]    
\end{itemize}       
\end{exas}

\begin{rema}
Let us note that tensoring with the  bicomplex $\ele_{-1}$ defines a \textit{shift functor} making the homotopy category of bicomplexes, defined by localizing at pluripotential weak equivalences, into a triangulated category. This shift has a homotopy inverse in the pluripotential sense, defined by tensoring with $\ele_1$. Moreover, there are pluripotential homotopy equivalences
\[\ele_1^{\otimes n}\simeq \ele_n\quad \text{ and }\quad \ele_{-1}^{\otimes n}\simeq \ele_{-n}\quad \text{ for }\quad n\geq 1.\]
\end{rema}

We define a functor
\[\Inf: \Ch \to \Bico.\]
Given a cochain complex
\[\dots \xrightarrow{d} C^{n-1} \xrightarrow{d} C^{n} \xrightarrow{d} C^{n+1} \xrightarrow{d} \dots,\]
we set
\[\Inf(C) = \left( \bigoplus_{n \in \mathbb{Z}} \ele_n \otimes C^n, \del, \delb \right),\]
where we consider $C^n$ as a bicomplex sitting in bidegree $(0,0)$. The differentials \(\del\) and \(\delb\) are defined by the following:
\begin{equation}
    \label{infdif-eq}
    \del((i,j)_n \otimes c) = \del(i,j)_n \otimes c + (-1)^{i+j+n} (i+1, j)_{n+1} \otimes d c,
\end{equation}
\[\delb((i,j)_n \otimes c) = \delb(i,j)_n  \otimes c + (-1)^{i+j+n} (i, j+1)_{n+1} \otimes d c.\]
A direct computation shows that this construction forms a bicomplex. We refer to this functor as the \textit{inflation} functor. Pictorially, the inflation functor looks as follows:

\[\fitText{\begin{tikzcd}[ampersand replacement=\&]
	\&\&\& {C^{2}} \& {C^2\oplus C^3} \& \iddots \\
	\&\&\& {C^{1}} \& {C^{1}\oplus C^{2}} \& {C^2\oplus C^3} \\
	{C^{-3}} \& {C^{-2}} \& {C^{-1}} \& {C^{0}} \& {C^{1}} \& {C^{2}} \\
	{C^{-3}\oplus C^{-4}} \& {C^{-2}\oplus C^{-3}} \& {C^{-1}\oplus C^{-2}} \& {C^{-1}} \\
	\iddots \& {C^{-3}\oplus C^{-4}} \& {C^{-2}\oplus C^{-3}} \& {C^{-2}} \\
	\&\iddots \& {C^{-3}\oplus C^{-4}} \& {C^{-3}}
	\arrow[shift left, draw={blue}, from=1-4, to=1-5]
	\arrow[shift right, from=1-4, to=1-5]
	\arrow[from=2-4, to=1-4]
	\arrow[shift left, draw={blue}, from=2-4, to=2-5]
	\arrow[shift right, from=2-4, to=2-5]
	\arrow[shift right, draw={blue}, from=2-5, to=1-5]
	\arrow[shift left, from=2-5, to=1-5]
	\arrow[shift left, draw={blue}, from=2-5, to=2-6]
	\arrow[shift right, from=2-5, to=2-6]
	\arrow[from=3-1, to=3-2]
	\arrow[from=3-2, to=3-3]
	\arrow[from=3-3, to=3-4]
	\arrow[from=3-4, to=2-4]
	\arrow[from=3-4, to=3-5]
	\arrow[shift right, draw={blue}, from=3-5, to=2-5]
	\arrow[shift left, from=3-5, to=2-5]
	\arrow[from=3-5, to=3-6]
	\arrow[shift right, draw={blue}, from=3-6, to=2-6]
	\arrow[shift left, from=3-6, to=2-6]
	\arrow[shift right, from=4-1, to=3-1]
	\arrow[shift left, draw={blue}, from=4-1, to=3-1]
	\arrow[shift right, draw={blue}, from=4-1, to=4-2]
	\arrow[shift left, from=4-1, to=4-2]
	\arrow[shift left, draw={blue}, from=4-2, to=3-2]
	\arrow[shift right, from=4-2, to=3-2]
	\arrow[shift right, draw={blue}, from=4-2, to=4-3]
	\arrow[shift left, from=4-2, to=4-3]
	\arrow[shift left, draw={blue}, from=4-3, to=3-3]
    \arrow[shift right, from=4-3, to=3-3]
	\arrow[shift left,  from=4-3, to=4-4]
    \arrow[shift right, draw={blue}, from=4-3, to=4-4]
	\arrow[from=4-4, to=3-4]
	\arrow[shift left, draw={blue}, from=5-2, to=4-2]
	\arrow[shift right, from=5-2, to=4-2]
	\arrow[shift right, draw={blue}, from=5-2, to=5-3]
	\arrow[shift left, from=5-2, to=5-3]
	\arrow[shift left, draw={blue}, from=5-3, to=4-3]
	\arrow[shift right, from=5-3, to=4-3]
	\arrow[shift left, from=5-3, to=5-4]
	\arrow[shift right, draw={blue}, from=5-3, to=5-4]
	\arrow[from=5-4, to=4-4]
	\arrow[shift left, draw={blue}, from=6-3, to=5-3]
	\arrow[shift right, from=6-3, to=5-3]
	\arrow[shift left, from=6-3, to=6-4]
	\arrow[shift right, draw={blue}, from=6-3, to=6-4]
	\arrow[from=6-4, to=5-4]
\end{tikzcd}}\]
 Here, the blue arrows denote the differentials of the bicomplexes \(\ele_n\otimes C^n\). The black arrows are the "connecting" differentials of the inflation functor given by the differential \(d\) of \(C\).
 
\begin{exas} We compute the inflation functor on simplicial cochain complexes:
\begin{itemize}[leftmargin=*]
    \item Consider the cochain complex $C_*(\Delta^1)$
   \begin{equation*}
        \underset{-1}{\kk\langle c\rangle} \xrightarrow{d} \underset{0}{\kk\langle c_1\rangle\oplus\kk\langle c_2\rangle} \qquad d(c)=c_2-c_1.
    \end{equation*}
    $\Inf(C_*(\Delta^1))$ is given by
        \[
        \begin{tikzcd}
        \kk \arrow[r, "d"] & \kk\oplus\kk \\
        \kk\arrow[r, "-id"] \arrow[u,"id"] & \kk \arrow[u,"d"]
        \end{tikzcd},
        \]
    so that $\del\delb c = c_2-c_1.$
    \item Consider the cochain complex \(C:=C_*(\Delta^2)\),
   \begin{equation*}
        \underset{-2}{\kk\langle [012]\rangle} \xrightarrow{d_2} \underset{-1}{\kk\langle [01]\rangle\oplus\kk\langle [02]\rangle\oplus\kk\langle [12]\rangle}\xrightarrow{d_1} \underset{0}{\kk\langle [0]\rangle\oplus\kk\langle [1]\rangle\oplus\kk\langle [2]\rangle}.
    \end{equation*}
    \(\Inf(C_*(\Delta^2))\) is given by
    
    \[
    \begin{tikzcd}[ampersand replacement=\&]
	{\color{blue}\kk} \& {\color{red}\kk^{\oplus 3}} \& {\color{orange}\kk^{\oplus 3}} \\
	{\color{blue}\kk}  \& {{\color{blue}\kk} \oplus{\color{red}\kk^{\oplus 3}}} \& {\color{red}\kk^{\oplus 3}} \\
	\& {\color{blue}\kk}  \& {\color{blue}\kk}
	\arrow["{d_2}", from=1-1, to=1-2]
	\arrow["{d_1}", from=1-2, to=1-3]
	\arrow["\cdot2", color={blue}, from=2-1, to=1-1]
	\arrow["{\cdot(-1)}"', shift right, color={blue}, from=2-1, to=2-2]
	\arrow["{-d_2}", shift left, from=2-1, to=2-2]
	\arrow["{\cdot 1}", shift left=1, color={red}, from=2-2, to=1-2]
	\arrow["{d_2}"', shift left=-1, from=2-2, to=1-2]
	\arrow["{\cdot (-1)}"', shift right, color={red}, from=2-2, to=2-3]
	\arrow["{d_2}", shift left, from=2-2, to=2-3]
	\arrow["{d_1}", from=2-3, to=1-3]
	\arrow["\cdot1", shift left=1, color={blue}, from=3-2, to=2-2]
	\arrow["{-d_2}"', shift left=-1, from=3-2, to=2-2]
	\arrow["{\cdot(-2)}"', color={blue}, from=3-2, to=3-3]
	\arrow["{d_2}", from=3-3, to=2-3]
    \end{tikzcd}.
    \]
    Here, the blue bicomplex is \(\ele_{-2}\otimes C_{-2}\), the red bicomplex is \(\ele_{-1}\otimes C_{-1}\), and the orange bicomplex is \(\ele_{0}\otimes C_0=\kk\otimes C_0\cong C_0\), which sits in bidegree \((0,0)\).
\end{itemize}

\end{exas}
\begin{rema}
 Note that \(\Inf(C_*(\Delta^\bullet))\) is a simplicial bicomplex.
 Since $\Inf$ preserves weak equivalences, for each $k$,
 \(\Inf(C_*(\Delta^k))\)
 is pluripotentially acyclic:
 \[\A(\Inf(C_*(\Delta^k)))=0\quad\text{ and }\quad\BC(\Inf(C_*(\Delta^k)))=0.\]
 In particular, it is natural to define a \textit{pluripotential totalization functor}
\[\mathrm{Tot}:\Delta^{op}\Bico\lra \Bico\]
as the end 
\[\mathrm{Tot}(A^\bullet):=\int_\alpha A^\alpha\otimes \Inf(C_*(\Delta^\alpha))\]
of the functor 
$\Delta^{op}\times \Delta\to \Bico$ given by 
$(\alpha,\beta)\mapsto A^\alpha\otimes \Inf(C_*(\Delta^\beta)).$
\end{rema}

Observe that $\Inf$ preserves direct sums and that given an injection of cochain complexes $D \hookrightarrow C$, the canonical morphism
\begin{align*}
    \Inf(C) / \Inf(D) &\to \Inf(C/D) \\
    [(i,j)_n \otimes c] &\mapsto (i,j)_n \otimes [c]
\end{align*}
is an isomorphism. So the inflation functor $\Inf$ commutes with colimits. As $\Ch$ and $\Bico$ are locally presentable, it follows that $\Inf$ admits a right adjoint. In the following theorem, we show that the inflation functor and $\Bb$ form an adjunction. We present two proofs. The first is more elegant, while the second is more explicit and introduces constructions that will be used later.
\begin{theo}
\label{adj-theo}
    The inflation functor $\Inf$ is left adjoint to $\Sch$.
\end{theo}
\begin{proof}[Proof 1]
    Denote by $\langle \delta \rangle$ the \(\kk\)-linear category with objects $n \in \ZZ$ and morphisms
    \begin{equation*}
        \langle \delta \rangle(m,n) =  
        \begin{cases}
            \kk & \text{if } m - n = 0 \text{ or } 1, \\
            0 & \text{otherwise}.
        \end{cases}
    \end{equation*}
    Then, cochain complexes are linear functors
    on $\langle \delta \rangle$ with values in $\kk$-Mod, $\Ch = [\langle \delta \rangle, \kk\text{-Mod}]$. Similarly, denote by $\langle \del,\delb\rangle$ the category whose objects are pairs $(m,n) \in \ZZ^2$ and morphisms are
    \begin{equation*}
        \langle \del, \delb \rangle( (m,n), (m',n')) = \begin{cases}
            \del \kk & \text{if } m - m' = 1 \text{ and } n = n', \\
            \delb\kk & \text{if }n - n' = 1 \text{ and } m = m', \\
            \mathrm{id}\kk & \text{if } n = n' \text{ and } m = m', \\
            \del\delb\kk & \text{if } m - m'=1 \text{ and } n - n'=1, \\
            0 & \text{otherwise}.
        \end{cases}
    \end{equation*}
    Define also the composition law so as to have $\del \delb = - \delb \del$. Then, bicomplexes are linear functors on $\langle \del, \delb \rangle$ with values in $\kk$-Mod. 
    
    Denote by $\iota : \langle \delta \rangle \to \Ch$ the Yoneda embedding and define $F = \Inf \circ \iota$. Since both $\Ch$ and $\Bico$ are presheaf categories and $\Inf$ preserves colimits, the right adjoint $R$ to $\Inf$ is given by
    \begin{equation*}
        R(B)_n  = \Hom_{\Bico}(F(n),B), \text{\qquad for } B \text{ in } \Bico.
    \end{equation*}
    Note that $F(n) = \Inf(D^n)$, where the disk $D^n$ is the cochain complex
    \begin{equation*}
        \cdots \to 0 \to \underset{n}{\kk} \to \underset{n+1}{\kk} \to 0 \to \cdots.
    \end{equation*}
    As a bigraded vector space, $\Inf(D^n) = (\ele_n \otimes \kk) \oplus (\ele_{n+1} \otimes \kk)$ and is freely generated by the elements in the lowest degree. 
    This implies that a morphism in $\Hom_{\Bico}(F(n),B)$ is determined by the image of the copies of $\kk$ in $\Inf(D^n)$ of lowest bidegrees. That is, it is determined by a choice of elements in
    \begin{align*}
        &\bigoplus_{i \geq 0}^{n} B^{i,n-i} \text{\quad \hspace{1ex} for } n \geq 0 \\
        &\bigoplus_{i = n}^{-1} B^{i,n-i-1} \text{\quad for } n < 0.
    \end{align*}
    It follows that, as graded $\kk$-modules, $R(B) = \Sch(B)$.
    Denoting by $\delta$ the generating morphism of $\langle \delta \rangle(n+1,n)$, one has that $F(\delta)$ simply projects onto $(\ele_{n+1} \otimes \kk)$. The differential of $R(B)$ is given by
    \begin{align*}
        \Hom_{\Bico}(F(n),B) &\xrightarrow[]{d} \Hom_{\Bico}(F(n+1),B) \\
        f &\mapsto f \circ F(\delta).
    \end{align*}  
    For $n = -1$, we have
    \[ 
    F(0) = \begin{tikzcd}
        \kk \arrow[r, "id"] & \kk \\
        \kk \arrow[r, "id"] \arrow[u,"id"] & \kk \arrow[u,"-id"]
    \end{tikzcd}
    \mathrm{\qquad} F(-1) = \begin{tikzcd}
        \kk \arrow[r, "id"] & \kk \\
        \kk \arrow[r, "-id"] \arrow[u,"id"] & \kk \arrow[u,"id"]
    \end{tikzcd}
    \]
    where the bottom left corners of $F(0)$ and $F(-1)$ sit in bidegree $(0,0)$ and $(-1,-1)$, respectively. A morphism $g \in \Hom_{\Bico}(F(0),B)$ is determined by $g^{0,0}$ since it commutes with differentials. Likewise, $f \in \Hom_{\Bico}(F(-1),B)$ is determined by $f^{\mep1,\mep1}$, so we can make the identifications
    \[ R^0(B) = \Hom_{\Bico}(F(0),B) \cong B^{0,0}, \qquad  R^{-1}(B) = \Hom_{\Bico}(F(-1),B) \cong B^{-1,-1}.\]
    Note that $F(\delta) : F(0) \to F(-1)$ projects onto the copy of $\kk$ in bidegree $(0,0)$. The image of $f \in \Hom_{\Bico}(F(-1),B)$ by the differential $d : R^{-1}(B) \to R^{0}(B)$ is then given by $(f \circ F(\delta))^{0,0} = f^{0,0}$. But $f^{0,0} = \del \delb f^{\mep1,\mep1}$. Hence, $d(f^{\mep1,\mep1}) = \del \delb f^{\mep1,\mep1}$, which is the differential of the complex $\Sch$.

    For \(n<-1\) the differentials of \(R(B)\) are given by 
    \[
    d(b^{n,\mep1},\dots,b^{\mep1,n})=\sum_{p+q=n-1}\frac{(p+1)\del b^{p,q}+(q+1)\delb b^{p,q}}{-n}
    \]
    and for \(n\geq 0\) are given by 
    \[
    d(b^{0,n},\dots,b^{n,0})=\sum_{\substack{p+q=n}}\frac{(p+1)\del b^{p,q}+(q+1)\delb b^{p,q}}{n}.
    \]
    We define a natural isomorphism from \(f\colon R \to\Sch\) by 
    \[
    f^n(b^{n,\mep1},\dots,b^{\mep1,n})=(-1)^{n+1}(\Le(\mep n,1)b^{n,\mep 1},\Le(\mep n,2)b^{n+1,\mep 2},\dots,\Le(1,\mep n)b^{\mep n,n}),
    \]
    for \(n<0\), and by
    \[
    f^n(b^{0,n},\dots,b^{n,0})=(\binom{n}{0}b^{0,n},\binom{n}{1}b^{1,n},\dots,\binom{n}{n}b^{n,0}),
    \]
    for \(n\geq 0\). Here, \(\binom{k}{n}\) denote the binomial coefficients whereas \(\Le(k,\ell)\) denote the entries of the Leibniz harmonic triangle \cite{OEIS_A003506}, explicitly given by \[\Le(k,\ell)=\frac{1}{\ell\binom{k}{\ell}}.\qedhere\]
\end{proof}

\begin{proof}[Proof 2]
We prove the adjunction by explicitly constructing the unit and the counit. We begin by defining the unit of the adjunction
    \[
    \eta:\mathrm{Id}_{\Ch}\to \Sch\circ \Inf.
    \]
    Note that, for a cochain complex \(C\), 
    \[
    \Sch\Inf(C)^n =
    \begin{cases}
    \displaystyle
    \bigoplus_{\substack{p+q=n\\ p,q\ge 1}}
     (p,q)_{\,n-1}\otimes C^{n-1}
    \;\oplus\;
    \bigoplus_{\substack{p+q=n\\ p,q\ge 0}}
     (p,q)_{\,n}\otimes C^{n}
    & \text{if }n\geq 0,
    \\[1.2em]
    \displaystyle
    \bigoplus_{\substack{p+q=n-1\\ p,q\le -1}}
     (p,q)_{\,n}\otimes C^{n}
    \;\oplus\;
    \bigoplus_{\substack{p+q=n-1\\ p,q\le -1}}
     (p,q)_{\,n-1}\otimes C^{n-1}
    & \text{if }n < 0.
    \end{cases}
    \]
    Below we present a picture of \(\Sch\circ \Inf(C)\):

\[
    \fitText{\begin{tikzcd}[ampersand replacement=\&]
	\&\&\& {C^{2}} \& {C^2\oplus C^3} \\
	\&\&\& {C^{1}} \& {C^{1}\oplus C^{2}} \& {C^2\oplus C^3} \\
	\textcolor{rgb,255:red,143;green,143;blue,143}{{C^{-3}}} \& \textcolor{rgb,255:red,143;green,143;blue,143}{{C^{-2}}} \& \textcolor{rgb,255:red,143;green,143;blue,143}{{C^{-1}}} \& {C^{0}} \& {C^{1}} \& {C^{2}} \\
	{C^{-3}\oplus C^{-4}} \& {C^{-2}\oplus C^{-3}} \& {C^{-1}\oplus C^{-2}} \& \textcolor{rgb,255:red,143;green,143;blue,143}{{C^{-1}}} \& { ^0} \& { ^1} \\
	\& {C^{-3}\oplus C^{-4}} \& {C^{-2}\oplus C^{-3}} \& \textcolor{rgb,255:red,143;green,143;blue,143}{{C^{-2}}} \\
	\&\& {C^{-3}\oplus C^{-4}} \& \textcolor{rgb,255:red,143;green,143;blue,143}{{C^{-3}}} \& {^{-1}} \\
	\&\&\&\& { ^{-2}}
	\arrow[shift left, color={rgb,255:red,92;green,92;blue,214}, from=1-4, to=1-5]
	\arrow[shift right, from=1-4, to=1-5]
	\arrow[from=2-4, to=1-4]
	\arrow[shift left, color={rgb,255:red,92;green,92;blue,214}, from=2-4, to=2-5]
	\arrow[shift right, from=2-4, to=2-5]
	\arrow[shift right, color={rgb,255:red,92;green,92;blue,214}, from=2-5, to=1-5]
	\arrow[shift left, from=2-5, to=1-5]
	\arrow[shift left, color={rgb,255:red,92;green,92;blue,214}, from=2-5, to=2-6]
	\arrow[shift right, from=2-5, to=2-6]
	\arrow[color={rgb,255:red,143;green,143;blue,143}, from=3-1, to=3-2]
	\arrow[color={rgb,255:red,143;green,143;blue,143}, from=3-2, to=3-3]
	\arrow[color={rgb,255:red,143;green,143;blue,143}, from=3-3, to=3-4]
	\arrow[from=3-4, to=2-4]
	\arrow[from=3-4, to=3-5]
	\arrow[shift right, color={rgb,255:red,92;green,92;blue,214}, from=3-5, to=2-5]
	\arrow[shift left, from=3-5, to=2-5]
	\arrow[from=3-5, to=3-6]
	\arrow[shift right, color={rgb,255:red,92;green,92;blue,214}, from=3-6, to=2-6]
	\arrow[shift left, from=3-6, to=2-6]
	\arrow[shift right, color={rgb,255:red,143;green,143;blue,143}, from=4-1, to=3-1]
	\arrow[shift left, color={rgb,255:red,139;green,155;blue,187}, from=4-1, to=3-1]
	\arrow[shift right, color={rgb,255:red,92;green,92;blue,214}, from=4-1, to=4-2]
	\arrow[shift left, from=4-1, to=4-2]
	\arrow[shift left, color={rgb,255:red,139;green,155;blue,187}, from=4-2, to=3-2]
	\arrow[shift right, color={rgb,255:red,143;green,143;blue,143}, from=4-2, to=3-2]
	\arrow[shift right, color={rgb,255:red,92;green,92;blue,214}, from=4-2, to=4-3]
	\arrow[shift left, from=4-2, to=4-3]
    \arrow[shift left, color={rgb,255:red,139;green,155;blue,187}, from=4-3, to=3-3]
    \arrow[shift right, color={rgb,255:red,143;green,143;blue,143}, from=4-3, to=3-3]
	\arrow["{\quad\circlearrowleft}"', from=4-3, to=3-4]
    \arrow[shift right, color={rgb,255:red,139;green,155;blue,187}, from=4-3, to=4-4]
    \arrow[shift left, color={rgb,255:red,143;green,143;blue,143}, from=4-3, to=4-4]
	\arrow[color={rgb,255:red,143;green,143;blue,143}, from=4-4, to=3-4]
	\arrow[shift left, color={rgb,255:red,92;green,92;blue,214}, from=5-2, to=4-2]
	\arrow[shift right, from=5-2, to=4-2]
	\arrow[shift right, color={rgb,255:red,92;green,92;blue,214}, from=5-2, to=5-3]
	\arrow[shift left, from=5-2, to=5-3]
	\arrow[shift left, color={rgb,255:red,92;green,92;blue,214}, from=5-3, to=4-3]
	\arrow[shift right, from=5-3, to=4-3]
	\arrow[shift left, color={rgb,255:red,143;green,143;blue,143}, from=5-3, to=5-4]
	\arrow[shift right, color={rgb,255:red,139;green,155;blue,187}, from=5-3, to=5-4]
	\arrow[color={rgb,255:red,143;green,143;blue,143}, from=5-4, to=4-4]
	\arrow[shift left, color={rgb,255:red,92;green,92;blue,214}, from=6-3, to=5-3]
	\arrow[shift right, from=6-3, to=5-3]
	\arrow[shift left, color={rgb,255:red,143;green,143;blue,143}, from=6-3, to=6-4]
	\arrow[shift right, color={rgb,255:red,139;green,155;blue,187}, from=6-3, to=6-4]
	\arrow[color={rgb,255:red,143;green,143;blue,143}, from=6-4, to=5-4]
\end{tikzcd}}
\]
The black piece together with the gray one constitute \(\Inf(C)\), whereas the "totalization" of the black part constitutes \(\Sch\circ \Inf(C)\). The numbers on the right hand side denote the degree of $\Bb$.

    We define:
    \[
    \eta_C(c)= 
    \begin{cases}
        \displaystyle
        \sum_{\substack{p+q=n\\p+q\geq 0}}\binom{n}{p}(p,q)_{n}\otimes c & \text{if }n\geq 0,\\[1.2em]
        \displaystyle
        \sum_{\substack{p+q=n-1\\p,q\leq -1}}(-1)^n\Le(-n,-p)(p,q)_{n}\otimes c & \text{if }n< 0.
    \end{cases}
    \]
    Here, \(\binom{k}{n}\) denote the binomial coefficients whereas \(\Le(k,\ell)\) denote the entries of the Leibniz harmonic triangle. 
    We now define the counit of the adjunction
    \[
        \varepsilon:\Inf\circ \Sch\to \mathrm{Id}_{\Bico}.
    \]
    Note that, for a bicomplex \(A\),
    \[
    \Inf\circ \Sch(A)^{p,q}=
    \begin{cases}
        \displaystyle
        \bigoplus_{\substack{r+s=p+q\\r,s\geq 0}}(p,q)_{p+q}\otimes A^{r,s}\oplus\bigoplus_{\substack{r+s=p+q-1\\r,s\geq 0}}(p,q)_{p+q-1}\otimes A^{r,s} & \text{ if }p,q\geq 0,\\[1.2em]
        \displaystyle
        \bigoplus_{\substack{r+s=p+q-1\\r,s\leq -1}}(p,q)_{p+q}\otimes A^{r,s}\oplus\bigoplus_{\substack{r+s=p+q\\r,s\leq -1}}(p,q)_{p+q+1}\otimes A^{r,s} & \text{ if }p,q< 0.
    \end{cases}
    \]
    We present a picture of \(\Inf\circ \Sch(A)\):
\[\fitText{\begin{tikzcd}[ampersand replacement=\&]
	\&\& {\bigoplus_{\substack{r+s=2\\r,s\geq 0}} A^{r,s}} \& {\bigoplus_{\substack{r+s=2\\r,s\geq 0}} A^{r,s}\oplus \bigoplus_{\substack{r+s=3\\r,s\geq 0}} A^{r,s}} \\
	\&\& {A^{1,0}\oplus A^{0,1}} \& {A^{1,0}\oplus A^{0,1}\oplus \bigoplus_{\substack{r+s=2\\r,s\geq 0}} A^{r,s}} \& {\bigoplus_{\substack{r+s=2\\r,s\geq 0}} A^{r,s}\oplus \bigoplus_{\substack{r+s=3\\r,s\geq 0}} A^{r,s}} \\
	{A^{\mep2,\mep1}\oplus A^{\mep1,\mep2}} \& {A^{\mep1,\mep1}} \& {A^{0,0}} \& {A^{1,0}\oplus A^{0,1}} \& {\bigoplus_{\substack{r+s=2\\r,s\geq 0}} A^{r,s}} \\
	{\displaystyle A^{\mep2,\mep1}\oplus A^{\mep1,\mep2}\oplus \bigoplus_{\substack{r+s=-4\\r,s\leq \mep1}} A^{r,s}} \& {A^{\mep1,\mep1}\oplus A^{\mep2,\mep1}\oplus A^{\mep1,\mep2}} \& {A^{\mep1,\mep1}} \\
	\& {A^{\mep2,\mep1}\oplus A^{\mep1,\mep2}\oplus \bigoplus_{\substack{r+s=-4\\r,s\leq -1}} A^{r,s}} \& {A^{\mep2,\mep1}\oplus A^{\mep1,\mep2}}
	\arrow[shift left, color={rgb,255:red,92;green,92;blue,214}, from=1-3, to=1-4]
	\arrow[shift right, from=1-3, to=1-4]
	\arrow[from=2-3, to=1-3]
	\arrow[shift left, color={rgb,255:red,92;green,92;blue,214}, from=2-3, to=2-4]
	\arrow[shift right, from=2-3, to=2-4]
	\arrow[shift right, color={rgb,255:red,92;green,92;blue,214}, from=2-4, to=1-4]
	\arrow[shift left, from=2-4, to=1-4]
	\arrow[shift left, color={rgb,255:red,92;green,92;blue,214}, from=2-4, to=2-5]
	\arrow[shift right, from=2-4, to=2-5]
	\arrow[from=3-1, to=3-2]
	\arrow["{\del\delb}", from=3-2, to=3-3]
	\arrow[from=3-3, to=2-3]
	\arrow[from=3-3, to=3-4]
	\arrow[shift right, color={rgb,255:red,92;green,92;blue,214}, from=3-4, to=2-4]
	\arrow[shift left, from=3-4, to=2-4]
	\arrow[from=3-4, to=3-5]
	\arrow[shift right, color={rgb,255:red,92;green,92;blue,214}, from=3-5, to=2-5]
	\arrow[shift left, from=3-5, to=2-5]
	\arrow[shift left, color={rgb,255:red,92;green,92;blue,214}, from=4-1, to=3-1]
	\arrow[shift right, from=4-1, to=3-1]
	\arrow[shift right, color={rgb,255:red,92;green,92;blue,214}, from=4-1, to=4-2]
	\arrow[shift left, from=4-1, to=4-2]
	\arrow[shift left, color={rgb,255:red,92;green,92;blue,214}, from=4-2, to=3-2]
	\arrow[shift right, from=4-2, to=3-2]
	\arrow[shift left, from=4-2, to=4-3]
	\arrow[shift right, color={rgb,255:red,92;green,92;blue,214}, from=4-2, to=4-3]
	\arrow["{\del\delb}"', from=4-3, to=3-3]
	\arrow[shift left, color={rgb,255:red,92;green,92;blue,214}, from=5-2, to=4-2]
	\arrow[shift right, from=5-2, to=4-2]
	\arrow[shift left, from=5-2, to=5-3]
	\arrow[shift right, color={rgb,255:red,92;green,92;blue,214}, from=5-2, to=5-3]
	\arrow[from=5-3, to=4-3]
\end{tikzcd}}
\]
Again, the blue arrows denote the differentials of the bicomplexes \(\ele_n\otimes \Sch(A)^n\). The black arrows are given by the differential of $\Bb$.

We define:
\[
\begin{aligned}
   \varepsilon^{p,q}_A\big((p,q)_{p+q}\otimes(a^{p+q,0},\dots,a^{0,p+q})+(p,q)_{p+q-1}\otimes(a^{p+q-1,0},\dots,a^{0,p+q-1})\big)=\\
   =p\Le(p+q,p)a^{pq}+\Le(p+q,p)\del a^{p-1,q}-\Le(p+q,q)\delb a^{p,q-1},
\end{aligned}
\]
if \(p,q\geq 0\) and
\[
\begin{aligned}
   \varepsilon^{p,q}_A\big((p,q)_{p+q+1}\otimes(a^{p+q+1,-1},\dots,a^{-1,p+q+1})+(p,q)_{p+q}\otimes(a^{p+q,-1},\dots,a^{-1,p+q})\big)=\\
   =(-1)^{p+q+1}\big(\frac{1}{\Le(-p-q-1,-p)}a^{pq}+\binom{-p-q-1}{-q-1}\del a^{p-1,q}-\binom{-p-q-1}{-p-1}\delb a^{p,q-1}),
\end{aligned}
\]
if \(p,q< 0\).

It remains to verify that
\[
\Inf\xrightarrow[]{\Inf\eta}\Inf\circ \Sch\circ \Inf\xrightarrow[]{\varepsilon\Inf}\Inf
\]
and
\[
    \Sch\xrightarrow[]{\eta\Sch}\Sch\circ\Inf\circ \Sch\xrightarrow[]{\Sch\varepsilon}\Sch
\]
are the identity morphisms \(1_\Inf\) and \(1_\Sch\), respectively.   

Let \(A\) be a bicomplex, and let \(\ov a =(a^{n,0},a^{n-1,1},\dots,a^{0,n})\) be an element in \(\Sch(A)^n\) for \(n\geq 0\). The map
\[
\eta_{\Sch}\colon\Sch(A)\to\Sch\circ \Inf\circ \Sch(A)
\]
gives
\[
\ov a\mapsto \sum_{\substack{p+q=n\\p,q\geq 0}}\binom{n}{p}(p,q)_n\otimes \ov a.
\]
Applying 
\[
\Sch(\varepsilon)\colon \Sch\circ \Inf\circ \Sch(A)\to \Sch(A)
\]
gives
\[
\sum_{\substack{p+q=n\\p,q\geq 0}}\binom{n}{p}(p,q)_n\otimes (a^{0,n},\dots,a^{n,0})\mapsto \sum_{\substack{p+q=n\\p,q\geq 0}}\binom{n}{q}\frac{1}{\binom{n}{p}}a^{p,q}=\ov a.
\]

Now let \(\ov a = (a^{n,-1},a^{n+1,-2},\dots,a^{-1,n})\) be an element in \(\Sch(A)^n\) for \(n<0\). The map
\[
\eta_{\Sch}\colon\Sch(A)\to\Sch\circ \Inf\circ \Sch(A)
\]
gives
\[
\ov a\mapsto \sum_{\substack{p+q=n-1\\p,q\leq -1}}(-1)^n\Le(-n,-p)(p,q)_{n}\otimes\ov a.
\]
Applying 
\[
\Sch(\varepsilon)\colon \Sch\circ\Inf\circ\Sch(A)\to \Sch(A)
\]
gives
\[
\sum_{\substack{p+q=n-1\\p,q\leq -1}}(-1)^n\Le(-n,-p)(p,q)_{n}\otimes\ov a\mapsto \sum_{\substack{p+q=n-1\\p,q\leq -1}}(-1)^{n+p+q+1}\frac{\Le(-n,-p)}{\Le(-p-q-1,-p)}a^{p,q}=\ov a.
\]

It is a matter of verification to check that the unit is a map of chain complexes, that the counit is a map of bicomplexes, and that the composition \(\Inf\xrightarrow[]{F\eta}\Inf\circ \Sch\circ \Inf\xrightarrow[]{\varepsilon\Inf}\Inf\) is the identity map. 
\end{proof}

\begin{rema}
    The binomial coefficients and the entries of the Leibniz harmonic triangle are related by the formula 
    \[
    \Le(r,c)=\frac{1}{c\binom{r}{c}}.
    \]
    Moreover, both Pascal and Leibniz triangles are defined by the same recursive rule: each entry in the triangle is obtained as the sum of the two nearest entries in the row below.
    \[
    \begin{matrix}
        & & \vdots & & & & \vdots & &\\
        1 & & 4 & & 6 & & 4 & & 1\\
        & 1 & & 3 & & 3 & & 1 &\\
        & & 1 & & 2 & & 1 & &\\
        & & & 1 & & 1 & & &\\
        & & & & 1 & & & & \\
        & & & \frac{1}{2} & &  \frac{1}{2} & & &\\
        & & \frac{1}{3} & & \frac{1}{6} & & \frac{1}{3} & & \\
        & \frac{1}{4} & & \frac{1}{12} & &\frac{1}{12} & & \frac{1}{4} &\\
        \frac{1}{5} & & \frac{1}{20} & & \frac{1}{30} & & \frac{1}{20} & & \frac{1}{5}\\
        & & \vdots & & & & \vdots & &
    \end{matrix}
    \]
\end{rema}

\begin{rema}
Note that, for bicomplexes concentrated in the first quadrant, $\Bb$ is just the usual totalization of bicomplexes. Therefore, in this particular setting, the inflation functor is 
 left adjoint to totalization.
\end{rema}

Consider $\Ch$ endowed with the projective model structure and \(\Bico\) with the pluripotential model structure of \cite{pluri}.

\begin{theo}\label{theo:Qadj}
    There is a Quillen adjunction 
    \[
    \begin{tikzcd}[ampersand replacement=\&]
	{\Inf:\Ch} \& {\Bico:\Sch}\;.
	\arrow[shift left, from=1-1, to=1-2]
	\arrow[shift left, from=1-2, to=1-1]
    \end{tikzcd}
    \]
\end{theo}  
\begin{proof}
     By Lemma \ref{lemm:SchWeakEq}, the functor $\Sch$, which is the right adjoint, preserves weak equivalences and fibrations, it follows that $\Inf \dashv \Sch$ is a Quillen adjunction.
\end{proof}

\begin{prop}
    The inflation functor preserves weak equivalences. Moreover, it is conservative.
\end{prop}
\begin{proof}
    The first part of the proposition is a corollary of Theorem \ref{theo:Qadj}. The second part can be proven as follows: let \(f\colon C\to D\) be a morphism of cochain complexes such that \(\Inf(f)\) is a pluripotential weak equivalence. By Theorem 1.21 of \cite{pluri}, it is a quasi-isomorphism with respect to the differential \(\del\) and \(\delb\).
    Note that the cochain complex \(\Inf(C)^{*,0}\) is isomorphic to the cochain complex \(C^*\), this is true for any cochain complex \(C\). It follows that \(f=\Inf(f)_{\mid\Inf(C)^{*,0}}\colon \Inf(C)^{*,0}\to \Inf(D)^{*,0}\) is a quasi-isomorphism of cochain complexes.
\end{proof}

\section{Monoidality of $\Bb$}
Denote by \( \Bico^{\doq} \) the subcategory of bicomplexes concentrated in the third quadrant, that is, \(B^{p,q}=0\) for \(p> 0\) or \(q>0\). Note that $\Inf(C)$ for a cochain complex $C$ in $\Ch^{\leq 0}$ lies in \( \Bico^{\doq} \).  It follows directly from Theorem \ref{adj-theo} that there is an adjunction
\[ \Inf: \Ch^{\leq 0} \leftrightharpoons \Bico^{\doq} : \Sch. \]

\begin{prop}
    \label{prop:colaxinf}
     The functor \[\Inf\colon \Ch^{\leq 0}\to \Bico^{\doq}\] is symmetric oplax monoidal with structure morphisms given by
     \begin{itemize}
         \item \(\phi=id: \Inf(\kk) \to \kk\), and
         \item for all $C$ and $D$ in $\Ch^{\leq 0}$,
         \begin{equation}
            \label{eq:colaxmorph}
             \phi_{C,D}:\Inf(C\otimes D)\to \Inf(C)\otimes \Inf(D),
         \end{equation}
         defined by
         \[
         \phi_{C,D}((i,j)_{k+\ell}\otimes c\otimes d)=\sum_{\substack{\alpha+\delta=i\\\beta+\gamma=j}}(-1)^{(\alpha+\beta)(\delta+\gamma+\ell)}(\alpha,\beta)_k\otimes c \otimes (\delta,\gamma)_\ell\otimes d,
         \]
         for \(c\in C^k\) and \(d\in D^\ell\).
     \end{itemize}
    Moreover, for all $C$ and $D$ in $\Ch^{\leq 0}$, the structure morphism \( \phi_{C,D} \) is a pluripotential weak equivalence.
\end{prop}
\begin{proof}
    Let $C,D\in \Ch^{\leq 0}$, first note that
    \[
    \Inf(C\otimes D)=\bigoplus_{n\in \ZZ}\bigoplus_{k+\ell=n}\ele_n\otimes C^k\otimes D^\ell
    \]
    and
    \[\Inf(C)\otimes \Inf(D)=\bigoplus_{n\in\ZZ}\bigoplus_{k+\ell=n}\ele_k\otimes C^k\otimes \ele_\ell\otimes D^\ell.\]
    
    For every $k,\ell\in\ZZ_{\leq 0}$ define a map of bicomplexes
    \[\iota^{k,\ell}:\ele_{k+\ell}\to\ele_k\otimes \ele_\ell\]
     given by
    \[\iota^{k,\ell}(i,j)_{k+\ell}=\sum_{\substack{\alpha+\delta=i\\\beta+\gamma=j}}(-1)^{(\alpha+\beta)(\delta+\gamma+\ell)}(\alpha,\beta)_k\otimes(\delta,\gamma)_\ell.\]

    Define then the following maps:
    \[\tau\circ(\iota^{k,\ell}\otimes id^{\otimes 2}): \ele_{k+\ell}\otimes C^k\otimes D^\ell\xrightarrow{\iota^{k,\ell}} \ele_k\otimes\ele_\ell \otimes C^k\otimes  D^\ell\xrightarrow[\cong]{\tau} \ele_k\otimes C^k\otimes \ele_\ell \otimes D^\ell, \]
    which yield a map of bicomplexes
    \[\phi_{C,D}:\Inf(C\otimes D)\to \Inf(C)\otimes \Inf(D).\]
    Note that there are no signs involved in $\tau$ since we consider $C^p$ a bicomplex sitting in degree $(0,0)$.
    The monoidal unit in $\Bico$ and in $\Ch$ is $\kk$. So we can define a morphism
    \[\phi=id: \Inf(\kk) \to \kk,\]
    and the unitality condition follows immediately. 

    It remains to check coassociativity, that is:
    \[(\phi_{C,D}\otimes id)\circ\phi_{C\otimes D,E}=(id\otimes \phi_{D,E})\circ\phi_{C,D\otimes E}:\Inf(C\otimes D\otimes E)\to \Inf(C)\otimes \Inf(D)\otimes \Inf(E).\]

    We need to show 
    \[
    (\iota^{r,s}\otimes id)\circ\iota^{r+s,t}=(id\otimes\iota^{s+t})\circ\iota^{r,s+t}
    \]
    for all \(r,s,t\in\ZZ_{\leq 0}\). On the one hand,
    \begin{align*}
        &(\iota^{r,s}\otimes id)\circ\iota^{r+s,t}(i,j)_{r+s+t}=\\
        &=\sum_{\substack{p_1+t_1=i\\ p_2+t_2=j}}(-1)^{d_{p}(d_{t}+t)}(\iota^{r,s}\otimes id)(p_1,p_2)_{r+s}\otimes(t_1,t_2)_t=\\
        &=\sum_{\substack{p_1+t_1=i\\ p_2+t_2=j}}\sum_{\substack{r_1+s_1=p_1\\ r_2+s_2=p_2}}(-1)^{d_{p}(d_{t}+t)+d_{r}(d_{s}+s)}(r_1,r_2)_r\otimes (s_1,s_2)_s\otimes(t_1,t_2)_t=\\
        &=\sum_{\substack{r_1+s_1+t_1=i\\ r_2+s_2+t_2=j}}(-1)^{(d_{r}+d_{s})(d_{t}+t)+d_{r}(d_{s}+s)}(r_1,r_2)_r\otimes(s_1,s_2)_s\otimes (t_1,t_2)_t.    
    \end{align*}
    Where \(d_{x}=x_1+x_2\). On the other hand,
        \begin{align*}
        &(id\otimes \iota^{s,t})\circ\iota^{r,s+t}(i,j)_{r+s+t}=\\
        &=\sum_{\substack{r_1+q_1=i\\ r_2+q_2=j}}(-1)^{d_{r}(d_{q}+q)}(id\otimes \iota^{s,t})(r_1,r_2)_r\otimes(q_1,q_2)_q=\\
        &=\sum_{\substack{r_1+q_1=i\\ r_2+q_2=j}}\sum_{\substack{s_1+t_1=q_1\\ s_2+t_2=q_2}}(-1)^{d_{r}(d_{q}+q)+d_{s}(d_{t}+t)}(r_1,r_2)_r\otimes(s_1,s_2)_s\otimes(t_1,t_2)_t=\\
        &=\sum_{\substack{r_1+s_1+t_1=i\\ r_2+s_2+t_2=j}}(-1)^{d_{r}(d_{s}+s+d_{t}+t)+d_{s}(d_{t}+t)}(r_1,r_2)_r\otimes(s_1,s_2)_s\otimes(t_1,t_2)_t.
    \end{align*}
    To prove it is symmetric, we want to check that the following diagram commutes 
    \[\begin{tikzcd}[ampersand replacement=\&]
	{\Inf(C\otimes D)} \& {\Inf(D\otimes C)} \\
	{\Inf(C)\otimes\Inf(D)} \& {\Inf(D)\otimes\Inf(C)\;.}
	\arrow["\cong"',"\tau_1", from=1-1, to=1-2]
	\arrow["{\phi_{C,D}}"', from=1-1, to=2-1]
	\arrow["{\phi_{D,C}}", from=1-2, to=2-2]
	\arrow["\cong"',"\tau_2", from=2-1, to=2-2]
    \end{tikzcd}\]
    Let \((i,j)_{k+\ell}\otimes c\otimes d\in \Inf(C\otimes D)\) where \(\mid c\mid=k\) and \(\mid d\mid=\ell\). On the one hand, we have
    \[
    \phi_{D,C} \circ \tau_1 = \sum_{\substack{p_1 + q_1 = i \\ p_2 + q_2 = j}} (-1)^{k\ell + d_q(d_p + k)} (q_1, q_2)_\ell \otimes d \otimes (p_1, p_2)_k \otimes c.
    \]
    On the other
    \[
    \tau_2\phi_{C,D}=\sum_{\substack{p_1+p_2=p\\q_1+q-2=q}}(-1)^{d_p(d_q+\ell)+d_pd_q}(q_1,q_2)_\ell\otimes d\otimes(p_1,p_2)_k\otimes c.
    \]
    If \(i + j = -k - \ell\), then \(d_p = -k\) and \(d_q = -\ell\), and clearly:
    \[
    \phi_{D,C} \circ \tau_1 = \tau_2 \circ \phi_{C,D}.
    \]
    If \(i + j = -k - \ell - 1\), then either \(d_p = -k\) and \(d_q = -\ell - 1\), or \(d_p = -k - 1\) and \(d_q = -\ell\). In both cases, the desired equality holds.

    To check that $\phi_{C,D}$ is a pluripotential weak equivalence, note that every complex is non-naturally isomorphic to a sum of complexes of the form:
    \[ \text{(points)} \quad \cdots \to 0 \to \kk \to 0 \to \cdots \qquad \text{and} \qquad \text{(lines)} \quad \cdots \to 0 \to \kk \xrightarrow[]{id} \kk \to 0 \to \cdots. \]
    Let us denote a point by $\bullet$ and a line by $\bullet \to \bullet$. By naturality of $\phi$ and the fact that $\Inf$ and $\otimes$ commute with direct sums, we only need to check that $\phi_{\bullet,\bullet}, \phi_{\bullet, \bullet \to \bullet}, \phi_{\bullet \to \bullet, \bullet}$ and $\phi_{\bullet \to \bullet,\bullet \to \bullet}$ are weak equivalences. Since $\bullet \to \bullet$ is a contractible cochain complex and $\Inf$ and $\otimes$ preserve weak equivalences, it follows that $\phi_{\bullet, \bullet \to \bullet}, \phi_{\bullet \to \bullet, \bullet}$ and $\phi_{\bullet \to \bullet,\bullet \to \bullet}$ are homotopically trivial. Denoting by $\bullet^k$ the point concentrated in degree $k\leq 0$, observe that $\Inf(\bullet^k) = \ele_{k}$ and
    \[\phi_{\bullet^k,\bullet^\ell} = \iota^{k,\ell} : \ele_{k+\ell} \to \ele_k \otimes \ele_\ell.\]
    The bicomplex $\ele_k \otimes \ele_\ell$ is isomorphic to the direct sum of $\ele_{k+\ell}$ with squares below each step. A square is a bicomplex as depicted below, where the differentials (the arrows) are isomorphisms,
    \[ \text{(square)} \quad \begin{tikzcd}
        \kk \arrow[r] & \kk \\
        \kk \arrow[r] \arrow[u] & \kk \arrow[u]
    \end{tikzcd}.
    \]
    For example for \(k=\ell=-1\) we have
\[\ele_{-1} \otimes \ele_{-1} \cong\begin{tikzcd}[ampersand replacement=\&]
	\kk \\
	\kk \& \kk^{\mep1,\mep1} \&\& \oplus \& \kk \& \kk^{\mep1,\mep1} \\
	\& \kk \& \kk \&\& \kk \& \kk
	\arrow[from=2-1, to=1-1]
	\arrow[from=2-1, to=2-2]
	\arrow[from=2-5, to=2-6]
	\arrow[from=3-2, to=2-2]
	\arrow[from=3-2, to=3-3]
	\arrow[from=3-5, to=2-5]
	\arrow[from=3-5, to=3-6]
	\arrow[from=3-6, to=2-6]
\end{tikzcd}\]
    where the superscripts denote the bidegrees. Moreover, the map $\iota^{k,\ell}$ is a split inclusion of $\ele_{k+\ell}$ into \(\ele_k\otimes\ele_\ell\). Since squares are homotopically trivial, the maps $\iota^{k,\ell}$ are weak equivalences. (cf. Lemma $1.9$ of \cite{pluri}).
\end{proof}

 We next describe the lax monoidal structure of the functor $\Sch$ when restricted to third-quadrant bicomplexes. 
\begin{theo}
    The functor 
    \[
    \Sch\colon\Bico^{\doq}\to \Ch^{\leq 0}
    \]
    is lax symmetric monoidal with structure morphisms given by:
    \begin{itemize}[leftmargin=*]
        \item \(\tilde{\phi}=id: \kk \to \Sch(\kk)\), and
        \item for all $A$ and $B$ in $\Bico^{\doq}$,
        \[
        \tilde{\phi}_{A,B}:\Sch(A)\otimes \Sch(B)\to\Sch(A\otimes B)
        \]
        defined on elements \(\ov a = (a^{-m,-1},a^{-m+1,-2},\dots,a^{-1,-m})\) in \(\Sch(A)^{-m}\) and \\\(\ov b = (b^{-n,-1},\dots,b^{-1,-n})\) in \(\Sch(B)^{-n},\) for \(m,n>1\), by
        \[
        \begin{aligned}
            &\tilde\phi_{A,B}^{-m-n}\big(\ov a\otimes\ov b\big)=\\
            &\sum_{\substack{\alpha+\beta=m\\ \gamma+\delta=n+1\\\alpha,\beta\ge 0,\gamma,\delta\geq 1}}
            (-1)^{n+1}\frac{\Le(m+n,\alpha+\gamma)}{\Le(n,\gamma)}
            \bigg(\binom{m-1}{\beta-1}\del a^{-\alpha-1,-\beta}-\binom{m-1}{\alpha-1}\delb a^{-\alpha,-\beta-1}\bigg)\otimes b^{-\gamma,-\delta}\\
            &\qquad-\sum_{\substack{\alpha+\beta=m+1\\ \gamma+\delta=n\\\alpha,\beta\ge 1,\gamma,\delta\geq 0}}
            \frac{\Le(m+n,\alpha+\gamma)}{\Le(m,\alpha)}
            a^{-\alpha,-\beta}\otimes
            \bigg(\binom{n-1}{\delta-1}\del b^{-\gamma-1,-\delta}-\binom{n-1}{\gamma-1}\delb b^{-\gamma,-\delta-1}\bigg.
        \end{aligned}
        \]
        If \(\ov a\in \Sch(A)^{0}\) or \(\ov b\in \Sch(B)^{0}\) we have
        \(\tilde\phi_{A,B}^{-m-n}\big(\ov a\otimes\ov b\big)=\ov a\otimes \ov b\).
    \end{itemize}
    Here, \(\Le(k,\ell)\) denotes the entries of the Leibniz harmonic triangle \cite{OEIS_A003506}, explicitly given by \(\Le(k,\ell)=\frac{1}{\ell\binom{k}{\ell}}\).
\end{theo}
\begin{proof}
Consider a colax monoidal functor
\[
F \colon (\Cc,\otimes,I_\Cc) \to (\Dd,\otimes,I_\Dd)
\]
between monoidal categories, with monoidal structure maps \(\nu:F(I_\Cc)\to I_\Dd\) and
\(\phi_{X,Y}:F(X\otimes Y)\to F(X)\otimes F(Y)\).
If \(F\) has a right adjoint
\[
G \colon \mathcal D \to \mathcal C,
\]
we can consider the adjoint
\[
\tilde{\nu} \colon I_{\mathcal D} \to G(I_{\mathcal C})
\]
of \(\nu\), and the natural map
\begin{equation*}
\tilde{\phi} \colon G(A)\otimes G(B) \to G(A \otimes B)
\end{equation*}
adjoint to the composite
\[
F\left(G(A)\otimes G(B)\right)
\xrightarrow{\phi_{G(A),G(B)}\ }
FG(A) \otimes FG(B)
\xrightarrow{\varepsilon_{A}\otimes\varepsilon_{B}}
A \otimes B.
\]

The map \(\tilde{\phi}\) can equivalently be defined as the composition
\begin{equation*}
\fitMath{
G(A)\otimes G(B)
\xrightarrow{\eta_{G(A) \otimes G(B)}}
GF\left(G(A) \otimes G(B)\right)
\xrightarrow{G(\phi_{G(A),G(B)}) }
G\left(FG(A) \otimes FG(B)\right)
\xrightarrow{G(\varepsilon_A\otimes\varepsilon_B) }
G\left(A \otimes B\right).}
\end{equation*}
Here \(\eta\) and \(\varepsilon\) denote, respectively, the unit and counit of the
adjunction. The pair \((\tilde{\nu},\tilde{\phi})\) defines a lax monoidal structure on the
right adjoint \(G\).
The dual construction is described in \cite{SchBro03}. If the left adjoint is symmetric colax monoidal, then the right adjoint is symmetric lax monoidal.

The unit and the counit of the adjunction
    \[
    \begin{tikzcd}[ampersand replacement=\&]
	{\Inf:\Ch^{\leq 0}} \& {\Bico^{\doq}:\Sch}
	\arrow[shift left, from=1-1, to=1-2]
	\arrow[shift left, from=1-2, to=1-1]
    \end{tikzcd}
    \]
given in the second proof of Theorem \ref{adj-theo} and the colax structure of \(\Inf\) given in Proposition \ref{prop:colaxinf} yield the explicit description of the lax monoidal structure on
\(\Sch\) when restricted to third-quadrant bicomplexes.
\end{proof}

\section{The infinity category of \texorpdfstring{\(\Bico\)}{Bico}}\label{sec:InftyCat}
In this section we describe three 
main strategies to build the $\infty$-category of bicomplexes with pluripotential weak equivalences:
through Dwyer--Kan localization, 
through a dg-enrichment, and through a simplicial enrichment. We show these are all equivalent.

\subsection*{A dg-enrichment}
Consider the canonical truncation functor for bicomplexes 
\[
\tau_{\doq}:\Bico\to\Bico^{\doq}
\]
as follows
\begin{equation}
\label{bitrunc-eq}
    \tau_{\doq}(B)=
    \begin{cases}
        B^{pq} & \text{ if }p,q<0,\\
        \ker(\del)^{pq} & \text{ if }p=0,q<0,\\
        \ker(\delb)^{pq} & \text{ if }p<0,q=0,\\
        \ker(\del)^{00} \cap \ker(\delb)^{00} & \text{ if }p=q=0\text{ and }\\
        0 & \text{ otherwise.}
    \end{cases}
\end{equation}
There is a Quillen adjunction
\[
i:\Bico^{\doq}\rightleftharpoons\Bico:\tau_{\doq},
\]
where \(i\) denotes the inclusion functor. From this adjunction we obtain that \(\tau_{\doq}\) is a lax monoidal functor, since \(i\) is oplax. Consider the following composite of lax monoidal functors
\[\Ee:\Bico\xra{\tau_{\doq}}\Bico^{\doq}\xra{\Sch}\Ch^{\leq 0}.\]

The category \(\Bico\) is self enriched, so we can apply the above composite of functors to Hom-objects. We denote by \(\Bico^\mathrm{Ch}\)  the
\textit{dg-category of bicomplexes} that has as objects bicomplexes, and for every $A$ and $B$ in $\Bico$, the Hom-object $[A,B]$ in $\Ch$ is given by
    \[ [A,B] = \Ee( \underline{\Hom}_{\Bico}(A,B)).\]

Applying the dg-nerve construction (see section $1.3.1$ of \cite{HA}), we get an infinity category $N_{dg}(\Bico^{\mathrm{Ch}})$.

\subsection*{A simplicial enrichment}
Now consider the Dold-Kan correspondence $N \dashv \Gamma$ and the forgetful functor
\[ 
\Ch^{\leq 0} \xrightarrow[]{\Gamma} s \mathrm{Mod}_\kk \xrightarrow[]{U} \mathrm{sSet}.
\]
In the same way, we get a \emph{simplicial category of bicomplexes} \(\Bico^\mathrm{sSet}\), whose objects are bicomplexes and for every $A$ and $B$ in $\Bico$, the Hom-object $[A,B]$ in \(\mathrm{sSet}\) is given by 
\[ 
U \Gamma \Ee( \underline{\Hom}_{\Bico}(A,B)).
\]

    The $n$-simplices of the Hom-objects \([A,B]\) in \(\Bico^{\mathrm{sSet}}\) are given by
    \begin{align*}
        U \Gamma \Ee( \underline{\Hom}_{\Bico}(A,B))_n &= \Hom_{\mathrm{sSet}}(\Delta^n, U \Gamma \Ee( \underline{\Hom}_{\Bico}(A,B))) \\
        &\cong \Hom_{\mathrm{sMod}_\kk}( \kk \Delta^n, \Gamma \Ee( \underline{\Hom}_{\Bico}(A,B))) \\
        &\cong \Hom_{\Ch^{\leq 0}}(N(\kk \Delta^n),  \Ee( \underline{\Hom}_{\Bico}(A,B))) \\
        &\cong \Hom_{\Bico}(\Inf(N(\kk \Delta^n)),\underline{\Hom}_{\Bico}(A,B)) \\
        &\cong \Hom_{\Bico}(\Inf(N(\kk \Delta^n))\otimes A, B).
    \end{align*}
    This implies that the functor \(B\mapsto [A,B]_n\) is represented by the element \(A\otimes \Inf(N(\kk\Delta^n))\). Moreover, since \(\Bico\) is a \(\Bico\)-enriched model category and all the functors involved in the simplicial enrichment are right Quillen, this enrichment satisfies the pullback-power axiom (M7 condition in Definition 9.1.6 in \cite{Hir03}) with respect to the pluripotential model structure. We obtain that \(\Bico^{\mathrm{sSet}}\) is a \emph{weak simplicial model category}, see \cite{Hi16}.

Denote by $N_{hc}(\Bico^{\mathrm{sSet}})$ the infinity category obtained by applying the homotopy coherent nerve.

\subsection*{Comparison with Dwyer--Kan localization}

Denote by $\Bico[W^{-1}]$ the $\infty$-category obtained by Dwyer-Kan localization with respect to pluripotential weak equivalences. 

\begin{theo}\label{bicoinf-theo}
    There is a chain of equivalences of \(\infty\)-categories 
    \[
    \Bico[W^{-1}] \simeq N_{hc}(\Bico^{\mathrm{sSet}}) \simeq N_{dg}(\Bico^{\mathrm{Ch}}).
    \]
\end{theo}
\begin{proof}
    The category \(\Bico\) is a locally presentable and cofibrantly generated model category. Moreover, \({\Bico^{\mathrm{sSet}}}\) is a weak simplicial model category. By Proposition 1.4.3 in~\cite{Hi16}, there is an equivalence of $\infty$-categories 
    \[
    \Bico[W^{-1}]\simeq N_{hc}(({\Bico^{\mathrm{sSet}}})_{cf}),
    \] 
    where \({(\Bico^{\mathrm{sSet}}})_{cf}\) denotes the subsimplicial category of cofibrant and fibrant objects, but every object in \(\Bico\) is fibrant and cofibrant.
    Lastly, by Proposition 1.3.1.17 of \cite{HA}, there is an equivalence of \(\infty\)-categories
    \[N_{hc}(\Bico^{\mathrm{sSet}})\simeq N_{dg}(\Bico^{\mathrm{Ch}}). \qedhere\]   
\end{proof}

\section{Real Bicomplexes}

As mentioned in the introduction, a natural source
for bicomplexes is the category of complex manifolds, through the complexified de Rham algebra of differential forms. By construction, in this setting, bicomplexes are defined over $\CC$ but have an underlying real structure. In this section, we adapt all the previous results to this setting.

\begin{defi}
    The category $\RR \Bicoc$, of $\RR$-bicomplexes, has as objects bicomplexes together with a $\CC$-antilinear involution $\sigma : A \to A$ such that $\sigma(A^{p,q}) = A^{q,p}$ and $\sigma \del \sigma = \delb$. The morphisms of $\RR \Bicoc$ are morphisms of bicomplexes that commute with the involutions.
\end{defi}
\begin{rema}
    An object of $\RR \Bicoc$ is equivalently defined as a graded $\RR$-complex $(A,d)$ such that the complexification $A \otimes \CC$ admits a bigrading and its differential splits into a sum $d = \del + \delb$, where $\del$ has bidegree $(1,0)$, $\delb$ has bidegree $(0,1)$, and $\delb$ is the conjugate of $\del$.
\end{rema}
The tensor product and the internal hom of $\RR$-bicomplexes are defined like in $\Bicoc$ with the involutions given by
\begin{align*}
\sigma : (A \otimes B)^{p,q} &\longrightarrow (A \otimes B)^{q,p}
&\qquad
\sigma : \underline{\Hom}_{\Bicoc}(A,B)^{p,q} &\longrightarrow \underline{\Hom}_{\Bicoc}(A,B)^{q,p} \\
a \otimes b &\longmapsto \sigma_A(a) \otimes \sigma_B(b)
&\qquad
f &\longmapsto \sigma_B \, f \, \sigma_A
\end{align*}

These make $\RR \Bicoc$ into a closed symmetric monoidal category.
\begin{defi}
    An \textit{$\RR$-pluripotential homotopy} between morphisms $f,g \in \Hom_{\RR \Bicoc}(A,B)$ is a morphism $h : A \to B$ of bidegree $(-1,-1)$ such that
    \[ [\del,[\delb,h]] = f-g \]
    and such that $\sigma h \sigma = -h$. If there exists a homotopy between $f$ and $g$, we write $f \sim g$. A morphism $f \in \Hom_{\RR \Bicoc}(A,B)$ is said to be an \textit{$\RR$-pluripotential homotopy equivalence} if there exists $g \in \Hom_{\RR \Bicoc}(B,A)$ such that $fg \sim id_B$ and $gf \sim id_A$.
\end{defi}
The class of $\RR$-pluripotential homotopy equivalences coincides with the class of morphisms of $\RR$-bicomplexes that are pluripotential weak equivalences. By Theorem 1.7 of \cite{pluri}, there is a model structure in $\RR \Bicoc$ with pluripotential weak equivalences as weak equivalences. Cofibrations are given by injective morphisms and fibrations are the surjective maps. The tensor product and internal Hom endow \(\RR\Bicoc\) with the structure of a symmetric monoidal model category, and hence \(\RR\Bicoc\) is a \(\RR\Bicoc\)-enriched model category.

We now define the variants of the inflation and $\Bb$ functors for $\RR$-bicomplexes. Consider the family of $\RR$-bicomplexes $\{\leftindex_\RR{\ele}_n\}_{n\in \ZZ}$ defined as in equation (\ref{eles-eq}) with differentials as in (\ref{eledif-eq}) and involution $\sigma : \leftindex_\RR{\ele}_n^{i,j} \to \leftindex_\RR{\ele}^{j,i}$ given by
\begin{align}
    \label{eleinv-eq}
    \sigma(i,j)_n = \begin{cases}
        (j,i)_n & \text{if } i+j = n \\
        -(j,i)_n & \text{if } i+j = n -1 \\
        -(j,i)_n & \text{if } i+j = n +1.
    \end{cases}
\end{align}
\begin{defi}
    The \textit{$\RR$-inflation} functor 
    \[ \Inf_\RR : \mathrm{Ch}_\RR \to \RR \Bicoc\]
    on a cochain complex $C \in \mathrm{Ch}_\RR$ is defined by setting
    \[ \Inf_\RR(C) = \Bigg( \bigoplus_{n \in \ZZ} \leftindex_\RR{\ele}_n \otimes (C \otimes \CC)^n, \del, \delb \Bigg),\]
    where we see $(C\otimes \CC)^n$ as a bicomplex sitting in bidegree $(0,0)$. The differentials are defined as in (\ref{infdif-eq}) and the involution is given by
    \[ (i,j)_n \otimes (c \otimes z) \mapsto \sigma{(i,j)}_n \otimes (c \otimes \overline{z}).\]
\end{defi}
The category $\RR \Bicoc$ is locally presentable and, like the previous inflation functor, $\Inf_\RR$ preserves colimits. Hence, it admits a right adjoint.
\begin{defi}
    The functor $\Bb_\RR\colon\RR \Bicoc \to \mathrm{Ch}_\RR$ is defined on an $\RR$-bicomplex $A$ as
    \[ \Bb_\RR^k(A) = \begin{cases}
        -i \cdot A_\RR^{k-1} \cap \bigoplus_{\substack{r+s=k-1\\r<0,s<0}} A^{r,s} &\text{if }k\leq -1\\
        ~\\
        A_\RR^k \cap \bigoplus_{\substack{r+s=k\\r\geq 0,s\geq 0}} A^{r,s} &\text{if }k\geq 0,
    \end{cases}\]
    where $A_\RR$ is the Tot complex of the fixed points by the involution of $A$. The differential is given by:
    \[
    \cdots\overset{\pr\circ d}{\longrightarrow}\Bb_\RR^{-2}(A)\overset{\pr\circ d}{\longrightarrow}\Bb_\RR^{-1}(A)\overset{ \del\delb}{\longrightarrow}\Bb_\RR^{0}(A)\overset{d}{\longrightarrow}\Bb_\RR^{1}(A)\overset{d}{\longrightarrow}\cdots
    \]
\end{defi}
\begin{theo}
    \label{realadj-theo}
    The functor $\Inf_\RR$ is left adjoint to $\Bb_\RR$.
\end{theo}
\begin{proof}
    The proof follows as in Theorem \ref{adj-theo}. The only difference is that we now identify $\RR \Bicoc$ with the category of $\kk$-linear functors $[ \langle \del, \delb \rangle_\RR, \CC\textrm{-}\mathrm{Mod}]$, where $\langle \del, \delb \rangle_\RR$ is the following category. Its objects are pairs $(m,n) \in \ZZ^2$ and morphisms are
    \begin{equation*}
        \langle \del, \delb \rangle_\RR( (m,n), (m',n')) = \begin{cases}
            \del \kk & \text{if } m - m' = 1 \text{ and } n = n' \\
            \delb\kk & \text{if } n - n' = 1 \text{ and } m = m' \\
            \del \delb \kk & \text{if } m - m' = 1 \text{ and } n - n' = 1 \\
            \mathrm{id} \kk \oplus \sigma & \text{if } m = m' = n = n' \\
            \mathrm{id}\kk & \text{if } n = n' \text{ and } m = m' \\
            \sigma \kk  &  \text{if } m = n' \text{ and } m' = n\\
            0 & \textit{otherwise}.
        \end{cases}
    \end{equation*}
    We define the composition so as to have $\del \delb = - \delb \del$ and $\sigma \sigma = \mathrm{id}$. Then, we define $F = \Inf_\RR \circ \iota$, where $\iota : \langle \delta \rangle \to \mathrm{Ch}_\RR$ is the Yoneda embedding and $\langle \delta \rangle$ is defined as in Theorem \ref{adj-theo}. The right adjoint $R$ to $\Inf_\RR$ is given by
    \[ R(A)^n = \Hom_{\RR \Bicoc}(F(n),A). \]
    We have 
    \[F(n) = (\ele_n \otimes \CC) \oplus (\ele_{n+1} \otimes \CC),\] its differentials are all non-zero multiples of the identity, and the involution is as in (\ref{eleinv-eq}). Since any morphism in $\Hom_{\RR \Bicoc}(F(n),A)$ has to commute with the involution, it is determined by the images of the fixed points of the copies of $\CC$ in lowest bidegrees. Note that for $n \geq 0$, the elements of $\leftindex_\RR{\ele}_n$ in lowest bidegree are $(p,q)_n$ for $p+q = n$, and the space of fixed points of the involution is generated by
    \[ (p,q)_n + (q,p)_n. \]
    While for $n < 0$, the elements of lowest bidegree are $(p,q)_n$ for $p+q = n-1$, and the fixed points are generated by
    \[ -i((p,q)_n + (q,p)_n). \] 
    Hence, identifying $f \in \Hom_{\RR \Bicoc}(F(n),A)$ with  the images $f((p,q) + (q,p))$, we get an isomorphism between $\Hom_{\RR \Bicoc}(F(n),A)$ and
    \begin{align*}
        &A_\RR^n \cap \bigoplus_{i \geq 0}^{n} A^{i,n-i} \text{\qquad \qquad \quad \hspace{0.1ex} for } n \geq 0, \\
        &-i \cdot A_\RR^{n-1} \cap \bigoplus_{i = n}^{-1} A^{i,n-i-1} \text{\quad for } n < 0.
    \end{align*}
    Thus, $R(A) = \Bb_\RR(A)$ as bigraded modules. The proof that there is an isomorphism of complexes between $R(A)$ and $\Bb_{\RR}(A)$ follows as in Theorem \ref{adj-theo}. 
\end{proof}
As a consequence, since $\Inf_\RR$ preserves weak equivalences and
cofibrations, it follows that $\Bb_\RR$ preserves weak equivalences and fibrations. Note as well that we have an adjunction

    \[
    \begin{tikzcd}[ampersand replacement=\&]
	{\Inf_\RR: \mathrm{Ch}_\RR^{\leq 0}} \& {\RR \Bicoc^{\doq} : \Bb_\RR}
	\arrow[shift left, from=1-1, to=1-2]
	\arrow[shift left, from=1-2, to=1-1]
    \end{tikzcd}
    \]

We now study the monoidal properties of the above adjunction.
\begin{prop}
    The functor $\Inf_\RR:\mathrm{Ch}_\RR^{\leq 0}\to \RR \Bicoc^{\doq}$ 
    is oplax monoidal with structure morphisms $\phi_{C,D}$, for each $C,D \in  \mathrm{Ch}_\RR^{\leq 0}$, given by the formulas in (\ref{eq:colaxmorph}). Moreover, each $\phi_{C,D}$ is an $\RR$-pluripotential weak equivalence.
\end{prop}
\begin{proof}
    The proof that $\Inf_\RR$ is oplax monoidal is like the one of Proposition \ref{prop:colaxinf}. It only remains to check that the structure morphisms are maps of $\RR$-bicomplexes. To do so, it is sufficient to check that the maps
    \[\iota^{pq}:\leftindex_\RR{\ele}_{p+q}\to\leftindex_\RR{\ele}_p\otimes \leftindex_\RR{\ele}_q\]
     given by
    \[\iota^{pq}(i,j)_{p+q}=\sum_{\substack{\alpha+\delta=i\\\beta+\gamma=j}}(-1)^{(\alpha+\beta)(\delta+\gamma+q)}(\alpha,\beta)_p\otimes(\delta,\gamma)_q\]
    are maps of $\RR$-bicomplexes. That is, 
    \[ \sigma \iota^{pq}\sigma = \iota^{pq}. \] 
    This follows from the observation that if $i+j = p+q$, then the action of $\sigma$ does not introduce any signs to either $(i,j)_{p+q}$ or the summands $(\alpha,\beta)_p \otimes (\delta,\gamma)_q$ of $\iota^{pq}(i,j)_{p+q}$ (see equation (\ref{eleinv-eq})). And if $i+j = p+q-1$, then for each summand $(\alpha,\beta)_p \otimes (\delta,\gamma)_q$, either $\alpha + \beta = p-1$ or $\delta + \gamma = q-1$, but not both. Then, the action of $\sigma$ on $(i,j)_{p+q}$ and on each summand introduces a minus sign.
\end{proof}
As a consequence of the previous proposition, $\Bb_\RR$ is a lax monoidal functor when restricted to third quadrant bicomplexes. Thus, we may endow $\RR \Bicoc$ with a dg-structure by defining the Hom-objects for $A,B \in \RR \Bicoc$ as
\[ [A,B] = \Ee_\RR(\underline{\Hom}_{\RR \Bicoc}(A,B)), \]
where $\Ee_\RR$ is the composition of functors
\[ \Ee_\RR : \RR \Bicoc \xrightarrow{\tau_{\doq}} \RR \Bicoc^{\doq} \xrightarrow[]{\Bb_\RR} \mathrm{Ch}_\RR^{\leq 0}\]
and $\tau_{\doq}$ is defined as in (\ref{bitrunc-eq}). Denote the resulting dg-category by $\RR \Bicoc^{\mathrm{Ch}}$. 

Similarly, we may obtain a simplicial category \(\RR\Bicoc^{\mathrm{sSet}}\) whose Hom-objects are given by
\[ [A,B] = U\Gamma\Ee_\RR(\underline{\Hom}_{\RR \Bicoc}(A,B)), \]
where \(U\Gamma\) is the composition functor
\[ 
\mathrm{Ch}^{\leq 0}_{\RR} \xrightarrow[]{\Gamma} s \mathrm{Mod}_\RR \xrightarrow[]{U} \mathrm{sSet}.
\]
given by the Dold-Kan correspondence and the forgetful functor.

As before, denote by \(\RR \Bicoc[W^{-1}]\) the \(\infty\)-category given by the Dwyer-Kan localization with respect to pluripotential weak equivalences. Precisely the same arguments as in Theorem \ref{bicoinf-theo} prove the theorem below.
\begin{theo}
    There is a chain of equivalences of \(\infty\)-categories
    \[
    \RR \Bicoc[W^{-1}] \simeq N_{hc}(\RR\Bicoc^{\mathrm{sSet}}) \simeq N_{dg}(\RR\Bicoc^{\mathrm{Ch}}).
    \]
\end{theo}

\bibliographystyle{alpha}
\bibliography{biblio}

@book{demailly,
  author = {Demailly, Jean-Pierre},
  title = {Complex Analytic and Differential Geometry},
  publisher = {Institut Fourier, Universit{\'e} de Grenoble I},
  year = {2012},
  note = {Available at \url{https://www-fourier.ujf-grenoble.fr/~demailly/manuscripts/agbook.pdf}}
}

@misc{Piovani,
      title={On the cohomology of the {B}igolin complex}, 
      author={Riccardo Piovani},
      year={2025},
      eprint={2405.20823},
      archivePrefix={arXiv},
      primaryClass={math.DG},
      url={https://arxiv.org/abs/2405.20823}, 
      howpublished = {Preprint, {arXiv}:2405.20823}
}

@misc{Jonas2,
      title={Strong formality of toric and homogeneous compact {K}\"ahler manifolds}, 
      author={Giovanni Placini and Jonas Stelzig and Leopold Zoller},
      year={2025},
      eprint={2510.17288},
      archivePrefix={arXiv},
      primaryClass={math.AT},
      url={https://arxiv.org/abs/2510.17288}, 
      howpublished = {Preprint, {arXiv}:2510.17288},
}

@misc{Jonas1,
 author = {Luc{\'{\i}}a Mart{\'{\i}}n-Merch{\'a}n and Jonas Stelzig},
 title = {Holomorphic linking numbers, {ABC} {Massey} products, and {Calabi}-{Yau} 3-folds},
 year = {2025},
 howpublished = {Preprint, {arXiv}:2512.02187},
 url = {https://arxiv.org/abs/2512.02187},
 arXiv = {arXiv:2512.02187}
}

@misc{schweizer,
      title={Autour de la cohomologie de {Bott}-{Chern}}, 
      author={Michel Schweitzer},
      year={2007},
      eprint={0709.3528},
      archivePrefix={arXiv},
      primaryClass={math.AG},
      url={https://arxiv.org/abs/0709.3528}, 
      howpublished = {Preprint, {arXiv}:0709.3528},
}

@article {pluri,
    AUTHOR = {Stelzig, Jonas},
     TITLE = {Pluripotential homotopy theory},
   JOURNAL = {Adv. Math.},
  FJOURNAL = {Advances in Mathematics},
    VOLUME = {460},
      YEAR = {2025},
     PAGES = {Paper No. 110038, 61},
      ISSN = {0001-8708,1090-2082},
   MRCLASS = {32C99 (18N40 57R19)},
  MRNUMBER = {4831838},
       DOI = {10.1016/j.aim.2024.110038},
       URL = {https://doi.org/10.1016/j.aim.2024.110038},
}

@article{BC,
author = {Raoul Bott and Shiing-Shen Chern},
title = {{Hermitian vector bundles and the equidistribution of the zeroes of their holomorphic sections}},
volume = {114},
journal = {Acta Mathematica},
number = {},
publisher = {Institut Mittag-Leffler},
pages = {71 -- 112},
year = {1965},
doi = {10.1007/BF02391818},
URL = {https://doi.org/10.1007/BF02391818}
}

@article{Aep,
author = {Alfred Aeppli},
title = {Some exact sequences in cohomology theory for {K}ähler manifolds},
volume = {12},
journal = {Pacific Journal of Mathematics},
number = {3},
publisher = {},
pages = {791–799},
year = {1962},
doi = {10.2140/pjm.1962.12.791},
URL = {}
}

@misc{HA, 
title={Higher Algebra}, 
author={Lurie, Jacob}, 
year={2014}, 
note = {Book draft. Available in \url{https://www.math.ias.edu/~lurie/}}}

@article{zbMATH03320556,
 author = {Bigolin, Bruno},
 title = {Gruppi di {Aeppli}},
 fjournal = {Annali della Scuola Normale Superiore di Pisa. Scienze Fisiche e Matematiche. III. Ser},
 journal = {Ann. Sc. Norm. Super. Pisa, Sci. Fis. Mat., III. Ser.},
 issn = {0036-9918},
 volume = {23},
 pages = {259--287},
 year = {1969},
 language = {Italian},
 url = {https://eudml.org/doc/83489},
 zbMATH = {3320556},
 Zbl = {0202.07601}
}

@article{StelzigSchRemarks,
 author = {Stelzig, Jonas},
 title = {Some remarks on the {Schweitzer} complex},
 fjournal = {Annales de l'Institut Fourier},
 journal = {Ann. Inst. Fourier},
 issn = {0373-0956},
 volume = {75},
 number = {1},
 pages = {35--47},
 year = {2025},
 language = {English},
 doi = {10.5802/aif.3645},
 zbMATH = {7983031},
 Zbl = {1558.32048}
}

@article {SchBro03,
    AUTHOR = {Schwede, Stefan and Shipley, Brooke},
     TITLE = {Equivalences of monoidal model categories},
   JOURNAL = {Algebr. Geom. Topol.},
  FJOURNAL = {Algebraic \& Geometric Topology},
    VOLUME = {3},
      YEAR = {2003},
     PAGES = {287--334},
      ISSN = {1472-2747,1472-2739},
   MRCLASS = {55U40 (18D10 18G30 18G35 55P43 55P62 55U35)},
  MRNUMBER = {1997322},
MRREVIEWER = {L.\ Gaunce\ Lewis, Jr.},
       DOI = {10.2140/agt.2003.3.287},
       URL = {https://doi-org.sire.ub.edu/10.2140/agt.2003.3.287},
}

@misc{OEIS_A003506,
  author       = {{OEIS Foundation Inc.}},
  title        = {{OEIS} Sequence {A003506}},
  howpublished = {\url{https://oeis.org/A003506}},
  year         = {2026},
  note         = {Accessed: 2026-01-16}
}

@book {Angella14,
    AUTHOR = {Angella, Daniele},
     TITLE = {Cohomological aspects in complex non-{K}\"ahler geometry},
    SERIES = {Lecture Notes in Mathematics},
    VOLUME = {2095},
 PUBLISHER = {Springer, Cham},
      YEAR = {2014},
     PAGES = {xxvi+262},
      ISBN = {978-3-319-02440-0; 978-3-319-02441-7},
   MRCLASS = {53C55 (32G07 32Q55 32Q60 53D18)},
  MRNUMBER = {3137419},
MRREVIEWER = {Anna\ M.\ Fino},
       DOI = {10.1007/978-3-319-02441-7},
       URL = {https://doi-org.sire.ub.edu/10.1007/978-3-319-02441-7},
}

@article{MiSte24,
    AUTHOR = {Milivojevi\'c, Aleksandar and Stelzig, Jonas},
     TITLE = {Bigraded notions of formality and
              {A}eppli-{B}ott-{C}hern-{M}assey products},
   JOURNAL = {Comm. Anal. Geom.},
  FJOURNAL = {Communications in Analysis and Geometry},
    VOLUME = {32},
      YEAR = {2024},
    NUMBER = {10},
     PAGES = {2901--2933},
      ISSN = {1019-8385,1944-9992},
   MRCLASS = {32C35 (55S30)},
  MRNUMBER = {4844967},
MRREVIEWER = {Giovanni\ Placini},
       DOI = {10.4310/cag.241231035705},
       URL = {https://doi-org.sire.ub.edu/10.4310/cag.241231035705},
}

@article {GFGS26,
    AUTHOR = {Garcia-Fernandez, Mario and Gonzalez Molina, Raul and Streets,
              Jeffrey},
     TITLE = {Pluriclosed flow and the {H}ull-{S}trominger system},
   JOURNAL = {Adv. Math.},
  FJOURNAL = {Advances in Mathematics},
    VOLUME = {485},
      YEAR = {2026},
     PAGES = {Paper No. 110699, 95},
      ISSN = {0001-8708,1090-2082},
   MRCLASS = {53E30 (53C55 58J35)},
  MRNUMBER = {4999594},
       DOI = {10.1016/j.aim.2025.110699},
       URL = {https://doi-org.sire.ub.edu/10.1016/j.aim.2025.110699},
}

@article {CiLivWhi25,
    AUTHOR = {Cirici, Joana and Livernet, Muriel and Whitehouse, Sarah},
     TITLE = {Model category structures on truncated multicomplexes for
              complex geometry},
   JOURNAL = {Bull. Lond. Math. Soc.},
  FJOURNAL = {Bulletin of the London Mathematical Society},
    VOLUME = {57},
      YEAR = {2025},
    NUMBER = {12},
     PAGES = {4163--4177},
      ISSN = {0024-6093,1469-2120},
   MRCLASS = {18N40 (18G40 32C99 32Q55)},
  MRNUMBER = {5007401},
       DOI = {10.1112/blms.70239},
       URL = {https://doi-org.sire.ub.edu/10.1112/blms.70239},
}

@article{CiGaSo25,
url = {https://doi.org/10.1515/coma-2025-0012},
title = {{\(A_\infty\)}-structures for {B}ott-{C}hern and {A}eppli cohomologies},
author = {Joana Cirici and Roger Garrido-Vilallave and Anna Sopena-Gilboy},
pages = {20250012},
volume = {12},
number = {1},
journal = {Complex Manifolds},
doi = {doi:10.1515/coma-2025-0012},
year = {2025},
lastchecked = {2026-02-07}
}

@phdthesis{SGtesi,
  author = {Sopena-Gilboy, Anna},
  title = {Pluripotential operadic calculus},
  school = {Universitat de Barcelona},
  year = {2026}
}

@article {AnTo15,
    AUTHOR = {Angella, Daniele and Tomassini, Adriano},
     TITLE = {On {B}ott-{C}hern cohomology and formality},
   JOURNAL = {J. Geom. Phys.},
  FJOURNAL = {Journal of Geometry and Physics},
    VOLUME = {93},
      YEAR = {2015},
     PAGES = {52--61},
      ISSN = {0393-0440,1879-1662},
   MRCLASS = {32C35 (53C55)},
  MRNUMBER = {3340173},
MRREVIEWER = {Cristian\ Ida},
       DOI = {10.1016/j.geomphys.2015.03.004},
       URL = {https://doi-org.sire.ub.edu/10.1016/j.geomphys.2015.03.004},
}

@article {TaTo17,
    AUTHOR = {Tardini, Nicoletta and Tomassini, Adriano},
     TITLE = {On geometric {B}ott-{C}hern formality and deformations},
   JOURNAL = {Ann. Mat. Pura Appl. (4)},
  FJOURNAL = {Annali di Matematica Pura ed Applicata. Series IV},
    VOLUME = {196},
      YEAR = {2017},
    NUMBER = {1},
     PAGES = {349--362},
      ISSN = {0373-3114,1618-1891},
   MRCLASS = {32Q55 (32C35 53C55)},
  MRNUMBER = {3600870},
MRREVIEWER = {Daniele\ Angella},
       DOI = {10.1007/s10231-016-0575-6},
       URL = {https://doi-org.sire.ub.edu/10.1007/s10231-016-0575-6},
}

@article {SfeTo22,
 author = {Sferruzza, Tommaso and Tomassini, Adriano},
 title = {Dolbeault and {Bott}-{Chern} formalities: deformations and \(\partial\overline\partial\)-lemma},
 fjournal = {Journal of Geometry and Physics},
 journal = {J. Geom. Phys.},
 issn = {0393-0440},
 volume = {175},
 pages = {19},
 note = {Id/No 104470},
 year = {2022},
 language = {English},
 doi = {10.1016/j.geomphys.2022.104470},
 zbMATH = {7485579},
 Zbl = {1487.32155}
}

@article {SfeTo24,
    AUTHOR = {Sferruzza, Tommaso and Tomassini, Adriano},
     TITLE = {{B}ott-{C}hern formality and {M}assey products on strong
              {K}\"ahler with torsion and {K}\"ahler solvmanifolds},
   JOURNAL = {J. Geom. Anal.},
  FJOURNAL = {Journal of Geometric Analysis},
    VOLUME = {34},
      YEAR = {2024},
    NUMBER = {11},
     PAGES = {Paper No. 348, 42},
      ISSN = {1050-6926,1559-002X},
   MRCLASS = {53C55 (22E25 53B35)},
  MRNUMBER = {4799037},
MRREVIEWER = {Antonella\ Nannicini},
       DOI = {10.1007/s12220-024-01764-w},
       URL = {https://doi-org.sire.ub.edu/10.1007/s12220-024-01764-w},
}

@article {Ru25,
    AUTHOR = {Rubini, Lapo},
     TITLE = {Some computations on trivial canonical-bundle solvmanifolds},
   JOURNAL = {J. Geom. Phys.},
  FJOURNAL = {Journal of Geometry and Physics},
    VOLUME = {216},
      YEAR = {2025},
     PAGES = {Paper No. 105586, 30},
      ISSN = {0393-0440,1879-1662},
   MRCLASS = {53C30 (32C35 32C36)},
  MRNUMBER = {4932593},
       DOI = {10.1016/j.geomphys.2025.105586},
       URL = {https://doi-org.sire.ub.edu/10.1016/j.geomphys.2025.105586},
}

@book{Hir03,
 author = {Hirschhorn, Philip S.},
 title = {Model categories and their localizations},
 fseries = {Mathematical Surveys and Monographs},
 series = {Math. Surv. Monogr.},
 issn = {0076-5376},
 volume = {99},
 isbn = {0-8218-3279-4},
 year = {2003},
 publisher = {Providence, RI: American Mathematical Society (AMS)},
 language = {English},
 zbMATH = {1860105},
 Zbl = {1017.55001}
}

@article{Hi16,
 author = {Hinich, Vladimir},
 title = {Dwyer-{Kan} localization revisited},
 fjournal = {Homology, Homotopy and Applications},
 journal = {Homology Homotopy Appl.},
 issn = {1532-0073},
 volume = {18},
 number = {1},
 pages = {27--48},
 year = {2016},
 language = {English},
 doi = {10.4310/HHA.2016.v18.n1.a3},
 zbMATH = {6618788},
 Zbl = {1346.18024}
}

\end{document}